\documentclass[10pt, a4paper]{amsart}

\usepackage[T1]{fontenc}
\usepackage[utf8]{inputenc}
\usepackage{amsmath, amssymb, amsfonts, amsthm}
\usepackage[margin=2.0cm]{geometry}
\usepackage[]{setspace}
\usepackage[usenames,dvipsnames]{color}
\usepackage[, unicode]{hyperref}
\hypersetup{
    urlcolor=blue,
    colorlinks=true,
   linkcolor= blue,
   citecolor=blue
}
\usepackage{paralist}
\usepackage{amsmath}
\numberwithin{equation}{section}
\usepackage{dirtytalk}
\usepackage{amssymb}
\usepackage{amsthm}
\usepackage{enumitem}
\usepackage{tikz-cd}
\usepackage[T1]{fontenc}
\usepackage[utf8]{inputenc}
\usepackage{mathrsfs}
\usepackage{xcolor}
\usepackage{hyperref}
\usepackage[toc,page]{appendix}
\hypersetup{
   colorlinks,
   citecolor=olive,
   filecolor=black,
   linkcolor=blue,
   urlcolor=blue
 }%
\usepackage{mathtools,halloweenmath}
\usepackage{mathbbol}
\theoremstyle{plain}
\newtheorem{propo}{Proposition}[section]
\newtheorem{lema}[propo]{Lemma}
\theoremstyle{plain}
\newtheorem{teorema}[propo]{Theorem}
\newtheorem{cor}[propo]{Corollary}
\newtheorem{fact}[propo]{Fact}

\theoremstyle{definition}

\newtheorem{defin}[propo]{Definition}

\newtheorem{remark}[propo]{Remark}

\newcommand{\ol}[1]{\overline{#1}}
\newcommand{\id}{id}

\renewcommand{\.}{\ldots}

\newcommand{\dl}{ \Big( \kern-0.6em{ \Big( }}
\newcommand{\dr}{ \Big) \kern-0.6em{ \Big) }}
\newcommand{\hs}[2]{#1\hspace{0.03cm}\dl\hspace{0.05cm}t^{#2}\hspace{0.05cm}\dr}

\newcommand{\sqdl}{ \Big[ \kern-0.55em{ \Big[ }}
\newcommand{\sqdr}{ \Big] \kern-0.55em{ \Big] }}

\DeclareMathOperator{\tp}{tp} 
\DeclareMathOperator{\Th}{Th} 
\DeclareMathOperator{\imdeg}{imp-deg} 
\DeclareMathOperator{\Frac}{Frac} 
\renewcommand{\char}{char} 
\DeclareMathOperator{\ac}{ac} 
\DeclareMathOperator{\res}{res} 
\DeclareMathOperator{\Mon}{Mon} 
\DeclareMathOperator{\cl}{cl} 

\def\a{\alpha}
\def\b{\beta}
\def\g{\gamma}
\def\G{\Gamma}
\def\d{\delta}
\def\D{\Delta}
\newcommand{\e}{\varepsilon}
\newcommand{\f}{\varphi}

\renewcommand{\th}{\theta}
\renewcommand{\t}{\tau}
 
\let\lw\l
\renewcommand\l{{\lambda}}

\def\k{\kappa}

\def\S{\Sigma}
\def\i{\iota}

\def\E{\exists} 
\def\A{\forall} 
\newcommand{\sii}{\longleftrightarrow}
\newcommand{\Sii}{\Longleftrightarrow}

\def\N{\mathbb{N}} 
\def\Z{\mathbb{Z}}

\def\F{\mathbb{F}}
\def\GG{\mathbb{G}}
\def\KK{\mathbb{K}} 
\def\kk{\mathbb{k}}

\def\AA{\mathfrak{A}} 
\def\BB{\mathfrak{B}} 
\def\MM{\mathfrak{M}} 
\def\NN{\mathfrak{N}} 
\def\LL{\mathscr{L}} 
\def\O{\mathcal{O}} 
\def\m{\mathfrak{m}} 

\def\Ind#1#2{#1\setbox0=\hbox{$#1x$}\kern\wd0\hbox to 0pt{\hss$#1\mid$\hss}
\lower.9\ht0\hbox to 0pt{\hss$#1\smile$\hss}\kern\wd0}

\def\Notind#1#2{#1\setbox0=\hbox{$#1x$}\kern\wd0\hbox to 0pt{\mathchardef
\nn=12854\hss$#1\nn$\kern1.4\wd0\hss}\hbox to
0pt{\hss$#1\mid$\hss}\lower.9\ht0 \hbox to
0pt{\hss$#1\smile$\hss}\kern\wd0}

\usepackage[utf8]{inputenc}
\usepackage{amsmath}
\usepackage{amsthm}
\usepackage{amssymb}
\usepackage{fancyhdr}
\usepackage{csquotes}
\usepackage{faktor}
\usepackage{enumitem}
\usepackage{dirtytalk}
\usepackage[english]{babel}
\newtheoremstyle{named}{}{}{}{}{}{}{.5em}{\normalfont{\thmnumber{#2}} \itshape\thmnote{#3.}}
\theoremstyle{named}
\newtheoremstyle{asd}{}{}{}{}{\scshape}{}{.5em}{\scshape{\thmname{#1}} \normalfont{\thmnumber{#2}}}
\theoremstyle{asd}
\fancypagestyle{mypagestyle}{%

\lhead{$\,$}
\rhead{\Cross}
}
\fancypagestyle{plain}{
\fancyfoot{}
\fancyfoot[LE,RO]{\thepage}
}
\usepackage{titlesec}
\titleformat{\section}
  {\centering\huge\scshape}{\scshape\thesection}{1em}{}
\titleformat{\subsection}
  {\centering\Large\scshape}{\scshape\thesubsection}{1em}{}
\titleformat{\subsubsection}{\large\bfseries}{\scshape\thesubsubsection}{0.5em}{}
\allowdisplaybreaks

\title[RUNNING TITLE]{\Large\rm RELATIVE QUANTIFIER ELIMINATION FOR SEPARABLE-ALGEBRAICALLY MAXIMAL KAPLANSKY FIELDS}
\author{PAULO ANDRÉS SOTO MORENO}
\thanks{\today}
\address{Universit\'{e} Paris Cit\'{e} and Sorbonne Universit\'{e}, CNRS, IMJ-PRG, F-75013 Paris, France}
\email{paulo.soto@imj-prg.fr}

\begin{document}

\begin{abstract}
Let $C$ be the class of separable-algebraically maximal equi-characteristic Kaplansky fields of a given imperfection degree, admitting an angular component map. We prove that the common theory of the class $C$ resplendently eliminates quantifiers down to the residue field and the value group, in a three sorted language of valued fields with a symbol for an angular component map and symbols for the parameterized lambda-functions. As a consequence, we obtain that equi-characteristic NIP and NIP$_n$ henselian fields with an angular component map resplendently eliminate field quantifiers in this language. We also prove that this elimination reduces existential formulas to existential formulas without quantifiers from the home sort. Finally, we draw several conclusions following the AKE philosophy for elements of the class $C$, including the usual AKE principles for $\equiv,\equiv_\E,\preceq,\preceq_\E$ and for relative decidability.    
\end{abstract}
\maketitle
\section{Introduction}

Quantifier elimination has been a central subject in the study of the model theory of valued fields, and has had a fruitful history. Some of the most notable results in this direction include:
\begin{itemize}
    \item Pas' relative quantifier elimination for henselian ac-valued fields of equi-characteristic 0, in the usual three sorted language of valued fields with an additional symbol for an angular component map, cf. \cite[Theorem 4.1]{pas},

    \item Basarab's relative quantifier elimination for henselian valued fields of characteristic 0, down to the so-called \emph{mixed-$k$-structures} for $k\geq0,$ cf. \cite[Theorem B]{bas},

    \item Bélair's relative quantifier elimination for unramified henselian valued fields of characteristic $(0,p),$ in a language with a compatible system of angular component maps, cf. \cite[Théorème 5.1]{bel}. Also, Bélair proves that Pas' result follows from Basarab's by means of an embedding lemma, cf. \cite[Théorème 4.1]{bel}.

    \item Kuhlmann's relative quantifier elimination for Algebraically Maximal Kaplansky equi-characteristic ac-valued fields, down to the so-called \emph{amc-structure}, cf. \cite[Theorem 2.6]{fvkqe},

    \item Flenner's relative quantifier elimination for characteristic 0 henselian fields, down to the so-called \emph{leading term structures}, cf. \cite[Proposition 4.3]{f}

    \item Halevi-Hason's relative quantifier elimination for Algebraically Maximal Kaplansky fields in a language with an angular component map, i.e. the Pas' language as above, cf. \cite[Corollary 4.6]{hh}, and

    \item Hong's quantifier elimination for Separably Closed Valued Fields of a given imperfection degree, in a one-sorted language of valued fields expanded by the parameterized $\l$-functions, cf. \cite[Theorem 4.12]{hong}.
\end{itemize}

The purpose of this article is to contribute to the study of Quantifier Elimination phenomena of valued fields, with a special focus on the class of equi-characteristic Separable-Algebraically Maximal Kaplansky fields. These fields appear naturally when studying classes of henselian valued fields satisfying classification-theoretic dividing lines like NIP and NIP$_{n},$ as given by the work of Anscombe and Jahnke in \cite{aj} and of Boissonneau in \cite{b}.  

The driving idea of this project builds on Kuhlmann's paper \cite{fvktf}, where he studies AKE principles for tame fields based on a number of \emph{embedding lemmas}. 
Following the same train of thought as his, our main theorems read as follows.

\begin{teorema}[Embedding Lemma, cf.~Theorem \ref{emblem}]
    Let $\MM,\NN$ be two models of $\texttt{SAMK}_{e}^{\l,\ac}$ and let $\AA$ be the $\LL$-substructure of $\MM$ generated by a $\l$-closed subring $A$ of $M.$ Suppose that $\MM$ is $\aleph_0$-saturated and that $\NN$ is $|M|^{+}$-saturated. 
    Let $\sigma:\kk(M)\to\kk(N)$ be a ring embedding and let $\rho:\GG(M)\to\GG(N)$ be and embedding of ordered groups. 
    If $\iota:A\to N$ is a ring embedding making $\Sigma=(\iota,\sigma|_{Av},\rho|_{vA}):\AA\to\NN$ into an $\LL$-embedding, then there is some ring embedding $\widetilde{\iota}:M\to N$ such that $\widetilde{\Sigma}:=(\widetilde{\iota},\sigma,\rho):\MM\to\NN$ is an $\LL$-embedding extending $\Sigma.$ 
\end{teorema}

\begin{teorema}[cf.~Proposition~\ref{rqe} and Corollary \ref{formbyform}]
    The theory $\texttt{SAMK}_{e}^{\lambda,\ac}$ resplendently eliminates field quantifiers in the three sorted language of valued fields expanded by symbols for an angular component map and the parameterized $\lambda$-functions.
\end{teorema}

When studying valued fields from a model-theoretic point of view, we will always work with three sorted languages, with sorts for the home field $\KK$, the residue field $\kk$ and the value group $\GG$. In this framework, \emph{resplendent elimination of field quantifiers} means that the corresponding theory has quantifier elimination with respect to the language obtained as the morleyzation of \emph{any} expansion of the language of the sort $\kk$ and the morleyzation of \emph{any} expansion of the language of the sort $\GG.$ See Remark \ref{respleq} for further details.


This document is organized as follows. In Chapter 2 we gather some preliminary results on top of which our main theorems are built. First, we give some background around separable field extensions, and explain how this concept is captured by the behavior of the associated $\lambda$-functions and the associated $p$-independence notion. Second, we include some general valuation theoretic facts about Separable Algebraically Maximal Kaplansky fields and Artin-Schreier extensions thereof. Third, we cite some fundamental results concerning pseudo-convergent sequences and immediate extensions, commonly known as Kaplansky Theory, and draw some conclusions in the separable-algebraically maximal case. Fourth, we recall some model-theoretic notations and facts in order to study existential formulas, and define the languages and the theories of $\texttt{SAMK}$ valued fields that we will consider from that point on. Finally, we recall some definitions and facts related to computability of first-order theories.   
In Chapter 3 we state and prove the Embedding Lemma \ref{emblem}, which is the key ingredient of the proof of our main theorem. In Chapter 4, we explain how, in our context, (resplendent, relative) quantifier elimination follows from an embedding extension problem, and use the Embedding Lemma to solve it. We conclude in Corollary \ref{nipn} that, in our language, NIP and NIP$_n$ equi-characteristic henselian $\ac$-valued fields admit resplendently eliminate field quantifier. Moreover, we see how the same technique yields that existential formulas reduce to existential formulas without quantifiers from the home sort, following a model-theoretic criterion.
Finally, in Chapter 5, we exploit the two reductions of Chapter 4 in order to get Ax-Kochen-Ershov-like results for models of $\texttt{SAMK}_{e}^{\lambda,\ac}.$ In particular, we use relative quantifier elimination to deduce the AKE-principles for $\equiv$ and $\preceq,$ and we use the reduction of existential formulas to deduce the AKE-principles for $\equiv_\E$ and $\preceq_\E.$ From this, we give a description of all completions of $\texttt{SAMK}_{e}^{\lambda,\ac},$ and also conclude that the residue field and the value group sorts are stably-embedded and orthogonal. Additionally, we study the decidability properties of this theory, where the focus is put into studying some many-one and Turing reductions down to the theories of the residue field and the value group, alongside their existential counterparts.

\section{Preliminaries} 

\begin{defin}
A valued field is called \emph{separable-algebraically maximal} if it does not admit any separable-algebraic immediate extension.
\end{defin}

By \cite[Theorem 2.4]{fvkp}, a valued field $(K,v)$ is separable-algebraically maximal if and only if it is henselian and $$\E x\in K\,\A y\in K\,\Big(vf(x)\geq vf(y)\Big)$$
for every separable polynomial $f\in K[X],$ which is axiomatizable in any of the standard languages of valued fields. The highlighted property is commonly known as \emph{extremality}, so a valued field is separable-algebraically maximal if and only if it is henselian and extremal for separable polynomials of one variable.

\begin{defin}
A valued field $(K,v)$ is called \emph{Kaplansky} if it is either of residue zero or of residue characteristic $p>0$ and
\begin{itemize}[wide]
    \item $vK$ is $p$-divisible,

    \item $Kv$ is perfect, and

    \item $Kv$ does not admit any proper separable algebraic extension of degree divisible by $p.$
\end{itemize}
\end{defin}
For example, $Kv=\bigcup_{n\geq1}\F_{p^{p^{n}}}$ satisfies these properties. By \cite[Theorem 5]{fvk}, a valued field $(K,v)$ of residue characteristic $p>0$ is Kaplansky if and only if $vK$ is $p$-divisible and $Kv$ is \emph{$p$-closed}, i.e. for any additive polynomial $f\in Kv[X]$ and any $\a\in Kv,$ the equation $f(x)=\a$ has a solution in $Kv.$ It is well known that if a field $K$ is infinite, then the set of additive polynomials over $K$ is exactly the class of polynomials of the form $f(X)=\sum_{i=0}^{n}a_iX^{p^{i}}$ with $a_0,\.,a_n\in K,$ implying that the property of $p$-closedness is first order axiomatizable in the language of rings $\LL_{ring}$, and the class of Kaplansky fields (of a given residue characteristic) is elementary. 
Altogether, the class $\texttt{SAMK}_{p}$ of separable-algebraically maximal Kaplansky fields of \emph{equi-characteristic} $p$ is elementary in the usual three-sorted language $\LL_{3s}$ of valued fields.

Unless stated otherwise, all fields mentioned in the rest of this chapter have characteristic $p>0$.

\subsection{Separability}

\begin{defin}
    Let $M$ be a field. We define the \emph{$p$-closure operation in $M$} as the closure operation on subsets $A$ of $M$ given by $$A\mapsto M^{p}(A).$$
\end{defin}

By \cite[Remark C.1.1]{tz}, this operation defines a pregeometry on $M,$ and likewise its relativizations\footnote{If $\cl:\mathcal{P}(M)\to\mathcal{P}(M)$ is a closure operation, then its relativization to a subset $S\subseteq M$ is the closure operation defined by $A\mapsto\cl(A\cup S).$} to arbitrary subsets of $M.$ If $M|F$ is a field extension, we say that a tuple $a$ in $M$ is \emph{$p$-independent in $M$ over $F$} if $$x\not\in M^{p}F(a\setminus\{x\})$$ for every $x\in a.$ We say that $a$ \emph{$p$-spans $M$ over $F$} if $M^{p}F(a)=M,$ and $a$ is a \emph{$p$-basis of $M$ over $F$} if it is $p$-independent in $M$ over $F$ and it $p$-spans $M$ over $F.$ If $F=\F_p,$ we say that $a$ is \emph{$p$-independent in $M$}, that it \emph{$p$-spans $M$} and that it is a \emph{$p$-basis of $M$} respectively. Finally, we define the \emph{imperfection degree of $M$ over $F$} to be the maximal (cardinal) length of a tuple in $M$ which is $p$-independent in $M$ over $F.$  
Note that a tuple $a$ of $M$ is $p$-independent in $M$ over $F$ (resp. $p$-spans $M$ over $F,$ is a $p$-basis of $M$ over $F$) if it is independent (resp. a generating set, a basis) with respect to the relativization to $F$ of the $p$-closure operation on $M.$ Also, the imperfection degree of $M$ over $F$ coincides with the dimension $\imdeg(M|F)$ of the $p$-closure operation on $M$ relativized to $F$. In this document, we define the \emph{Ershov degree} of $M$ over $F$ as $$\begin{cases}\imdeg(M|F)&\text{ if }\imdeg(M|F)\text{ is finite},\\\infty&\text{ otherwise.}\end{cases}$$ \emph{This is a purely formal definition that will only be used for axiomatizability purposes}, for we will be mainly interested in the imperfection degree rather than in Ershov degrees. If $\k\leq|M|$ is a cardinal, let $\Mon(\k)$ be the set of tuples $m=(m_\b)_{\b<\k}\in\{0,\.,p-1\}^{\k}$ whose support $\{\b<\k:m_\b\neq0\}$ is finite. For each cardinal $\k\leq|M|,$ each tuple $a=(a_\b)_{\b<\k}$ in $M$ of length $\k$ which is $p$-independent in $M$ over $F,$ and each $m\in\Mon(\k),$ define $$a^{m}:=\prod_{\b<\k}a_\b^{m_\b}=\prod_{\substack{\b<\k\\m_\b\neq 0}}a_\b^{m_\b}.$$  
Given one such tuple $a,$ the set $\{a^{m}:m\in\Mon(\k)\}$ is an $M^{p}F$-linear basis of $M^{p}F(a),$ cf. \cite[Lemma 2.3]{at}. This implies that for each $t\in M^{p}F(a)$ there are uniquely determined elements $\{\mu_{m}^{a}(t):m\in\Mon(\k)\}$ of $M^{p}F$ such that $$t=\sum_{m\in\Mon(\k)}\mu_{m}^{a}(t)\cdot a^{m},$$ thus defining the \emph{coordinate functions} $\mu_{m}^{a}:M^{p}F(a)\to M^{p}F.$ Unless stated otherwise, we will assume that $F$ is the prime field $\F_p$, so for each $m\in\Mon(\k)$ the coordinate function $\mu_{m}^{a}$ has domain $M^{p}(a)$ and co-domain $M^{p}.$ We now define the \emph{parameterized $\l$-functions} as the functions $\{\l_{m}^{a}:m\in\Mon(\k)\},$ with $\l_{m}^{a}:M^{p}(a)\to M,$ satisfying that $(\l_{m}^{a}(t))^{p}=\mu_{m}^{a}(t)$ for each $t\in M^{p}(a),$ i.e. satisfying $$t=\sum_{m\in\Mon(\k)}(\l_{m}^{a}(t))^{p}\cdot a^{m}.$$ 
\begin{remark}
\label{defl}
    Note that these conditions defining the parameterized $\l$-functions are elementary in the language of rings. 
\end{remark}

The following fact goes back at least as far as Mac Lane. We point to \cite{at} for more information. 

\begin{fact}[Cf. {\cite[Proposition 2.7]{at}}]
\label{sylvy}
Let $M|F$ be a field extension. The following statements are equivalent.
\begin{enumerate}[label*={\arabic*.}]
    \item $M|F$ is separable.

    \item For each tuple $b$ from $F$ which is $p$-independent in $M,$ we have that $M^{p}(b)\cap F=F^{p}(b).$

    \item For each tuple $b$ from $F$ which is $p$-independent in $M,$ each monomial $m\in\Mon(|b|)$ and each $a\in F\cap M^{p}(b),$ the element $\l_m^{b}(a)$ lies in $F.$
\end{enumerate}
\end{fact}

\begin{defin}
We say that $F$ is \emph{$\l$-closed in $M$} if statement 3 of Fact \ref{sylvy} holds. Note that this definition also makes sense if $F$ is just a subring of $M.$ If the ambient field $M$ is clear from the context, we say that $F$ is \emph{$\l$-closed.}  
Also, let $F$ and $N$ be two fields and let $\iota:F\to N$ be a ring embedding. We say that $\iota$ is \emph{separable} if $N|\iota F$ is separable.  
\end{defin}

Therefore, $F$ is $\l$-closed in $M$ if and only if $M|F$ is separable, by Fact \ref{sylvy}.

\begin{lema}
\label{cosepemb}
    Let $M|F$ be a separable field extension, let $N$ be another field and let $\iota:F\to N$ be a separable ring embedding. Then for any $b\in F^{n}$ which is $p$-independent in $M,$ any $a\in F\cap M^{p}(b)$ and any monomial $m\in\Mon(n),$ 
    $$\iota\left(\l_{m}^{b}(a)\right)=\l_{m}^{\iota b}(\iota a),$$ where the lambdas of the left hand side correspond to $M$ and the ones of the right hand side correspond to $N.$
\end{lema}
\begin{proof}
    If $b\in F^{n}$ is $p$-independent in $M,$ then it is trivially $p$-independent in $F.$ Therefore $\iota b$ is $p$-independent in $\iota F,$ and since $\iota$ is separable, $\iota b$ is $p$-independent in $N.$ Now let $a\in F\cap M^{p}(b),$ so that $$a=\sum_{m\in\Mon(n)}\left(\lambda_{m}^{b}(a)\right)^{p}\cdot b^{m}.$$
    Since $M|F$ is separable, we know by Fact \ref{sylvy} that $F$ is $\l$-closed in $M,$ i.e. $\lambda_{m}^{b}(a)\in F$ and $\iota\left(\lambda_{m}^{b}(a)\right)$ is well defined for all $m\in\Mon(n).$ Applying $\iota$ yields that 
    \begin{align*}
        \iota a&=\sum_{m\in\Mon(n)}\left(\iota\left(\lambda_{m}^{b}(a)\right)\right)^{p}\cdot (\iota b)^{m}\in N^{p}(\iota b)\\
        &=\sum_{m\in\Mon(n)}\left(\lambda_{m}^{\iota b}(\iota a)\right)^{p}\cdot (\iota b)^{m},
    \end{align*}
    where the second equality holds because $\iota b$ is a $p$-basis of $N^{p}(\iota b).$ Then, $p$-independence of $\iota b$ in $N$ implies that $$\iota\left(\l_{m}^{b}(a)\right)=\l_{m}^{\iota b}(\iota a)$$ for all $m\in\Mon(n),$ as wanted.
\end{proof}

\begin{lema}
\label{lcltofrac}
Let $M$ be a field and let $A\subseteq M$ be a $\l$-closed subring. Then $F=\Frac(A)$ is a $\l$-closed subfield of $M.$ 
\end{lema} 
\begin{proof}
    We want to prove that $\l_m^{b}(a)\in F$ for any $\k$-tuple $b$ from $F$ which is $p$-independent in $M,$ any $a\in F\cap M^{p}(b)$ and any monomial $m\in\Mon(\k).$  
    After renaming the indices of $m$, we may assume that $\k=n$ is a finite cardinal, given that $m\in\Mon(\k)$ has finite support. 
    Let $b=(b_1,\.,b_n)=(a_1/c_1,\.,a_n/c_n)$ and let $a=x/y,$ where $a_1,\.,a_n,$ $c_1,\.,c_n,$ $x,y\in A$ are all non-zero. 
    First, we claim that $d:=(c_1^{p-1}a_1,\.,c_n^{p-1}a_n)$ is a tuple in $A$ which is $p$-independent in $M.$ To see this, we may assume that $(c_1^{p-1}a_1,\.,c_l^{p-1}a_l),$ with $l\leq n,$ is a maximally $p$-independent sub-tuple of $d$ in $M,$ for which we will argue that $l=n.$ If this is not the case, we would have that
    $$c_n^{p}b_n=c_n^{p-1}a_n=\sum_{m\in\Mon(l)}\left(\l_m\cdot c_1^{m_1}\cdot\.\cdot c_l^{m_l}\right)^{p}\cdot b_1^{m_1}\cdot\.\cdot b_l^{m_l}$$
    for some elements $\l_m\in M,$ implying that 
    $$b_n=\sum_{m\in\Mon(l)}\left(\dfrac{\l_m\cdot c_1^{m_1}\cdot\.\cdot c_l^{m_l}}{c_n}\right)^{p}\cdot b_1^{m_1}\cdot\.\cdot b_l^{m_l}$$ and that $b$ is not $p$-independent in $M.$ 
    Second, we claim that $M^{p}(b)=M^{p}(d).$ Indeed, if $i\in\{1,\.,n\},$ then $$b_i=a_i/c_i=(1/c_i)^{p}\cdot c_i^{p-1}a_i\in M^{p}(d)\text{ and }c_i^{p-1}a_i=c_i^{p}b_i\in M^{p}(b).$$ 
    This implies that $b$ and $d$ are two $p$-bases of $M^{p}(b)=M^{p}(d).$ Then, for each monomial $m'\in\Mon(n),$ we have that $$d^{m'}=\sum_{m\in\Mon(n)}\left(\l_m^{b}(d^{m'})\right)^{p}\cdot b^{m}.$$ If $a\in F\cap M^{p}(b)=F\cap M^{p}(d),$ then 
    $$a=\sum_{m'\in\Mon(n)}\left(\l_{m'}^{d}(a)\right)^{p}\cdot d^{m'}=\sum_{m\in\Mon(n)}\left(\sum_{m'\in\Mon(n)}\l_{m'}^{d}(a)\cdot\l_m^{b}(d^{m'})\right)^{p}\cdot b^{m},$$
    hence $$\l_m^{b}(a)=\sum_{m'\in\Mon(n)}\l_{m'}^{d}(a)\cdot\l_m^{b}(d^{m'})$$ for all monomials $m\in\Mon(n).$ Therefore, the result of the Lemma follows once we prove that $\l_{m}^{d}(a)$ and $\l_m^{b}(d^{m'}),$ with $m,m'\in\Mon(n),$ are all elements of $F.$ First, by statement (3) of \cite[Lemma 2.15]{at}, $$\l_m^{b}(d^{m'})\in\F_p\left[b,\l_{\nu}^{b}(c_1^{p-1}a_1),\.,\l_{\nu}^{b}(c_n^{p-1}a_n):\nu\in\Mon(n)\right].$$ Then $\l_m^{b}(d^{m'})\in F$ if $\l_{\nu}^{b}(c_i^{p-1}a_i)\in F$ for all $\nu\in\Mon(n)$ and all $i\in\{1,\.,n\}.$ For such an index $i,$ let $\nu(i)=(\nu_1(i),\.,\nu_n(i))\in\Mon(n)$ be defined by $\nu_j(i)=\d_{ij},$ where $\d_{ij}$ is Kronecker's delta. We get that $$\l_\nu^{b}\left(c_i^{p-1}a_i\right)=\l_\nu^{b}\left(c_i^{p}b_i\right)=\begin{cases}
        c_i\text{ if }\nu=\nu(i),\\
        0\text{ otherwise},
    \end{cases}$$
    which is always an element of $F,$ as wanted. Finally, if $m\in\Mon(n)$ is any monomial and $a=x/y\in M^{p}(d),$ then $y^{p-1}x=y^{p}a\in M^{p}(d)$ too, so 
    $$\sum_{m\in\Mon(n)}\left(\l_m^{d}\left(\dfrac{x}{y}\right)\right)^{p}\cdot d^{m}=\dfrac{x}{y}=\dfrac{1}{y^{p}}\,y^{p-1}x=\sum_{m\in\Mon(n)}\left(\dfrac{\l_m^{d}\left(y^{p-1}x\right)}{y}\right)^{p}\cdot d^{m},$$
    implying that $$\l_m^{d}\left(a\right)=\dfrac{\l_m^{d}\left(y^{p-1}x\right)}{y}\in F,$$ as wanted.
\end{proof}

The result of the following fact is well known even for arbitrary pregeometries, cf. \cite[Remark C.15]{tz}.
\begin{fact}
\label{comppind}
Let $M|F$ be a field extension. Let $c$ be a tuple in $F$ which is $p$-independent in $F$ and let $a$ be a tuple in $M\setminus F$ which is $p$-independent in $M$ over $F.$ Then the concatenation $ac$ is $p$-independent in $M.$ Moreover, if $c$ is a $p$-basis of $F$ and $a$ is a $p$-basis of $M$ over $F,$ then $ac$ is a $p$-basis of $M.$
\end{fact}

\begin{lema}
\label{sepcorr}
    Let $M|F$ be a field extension, let $\iota:M\to N$ be a ring embedding and let $a$ be a tuple of $M\setminus F$ which is a $p$-basis of $M$ over $F.$ If the restriction of $\iota$ to $F$ is separable and $\iota a$ is $p$-independent in $N$ over $\iota F,$ then $\iota$ is separable.
\end{lema}
\begin{proof}
Let $c$ be a tuple in $F$ which is $p$-basis of $F.$ By Lemma \ref{comppind}, the concatenation $ac$ is a $p$-basis of $M.$ Since $\iota$ is an embedding, $\iota(ac)$ is a $p$-basis of $\iota M.$ Note that $\iota(ac)$ equals the concatenation $\iota(a)\iota(c).$ 
Now, we will conclude that the extension $N|\iota M$ is separable once we prove that $\iota(a)\iota(c)$ is $p$-independent in $N,$ by \cite[Lemma 2.2]{at}. 
To see this, note that $\iota(c)$ is a $p$-basis of $\iota F,$ and since $N|\iota F$ is separable, $\iota(c)$ is $p$-independent in $N.$ Since $\iota(a)$ is $p$-independent in $N$ over $\iota F,$ we get again by Lemma \ref{comppind} that $\iota(a)\iota(c)$ is $p$-independent in
$N,$ as wanted.
\end{proof}

The following lemma is well known, cf. \cite[Lemma 9, Example 2.13]{a}.

\begin{lema}
\label{sepcore}
    Let $M|F$ be a separable extension, let $a\in M$ and suppose that $F(a^{1/p^{\infty}})\subseteq M.$ Then $M|F(a^{1/p^{\infty}})$ is separable.
\end{lema}
\begin{proof}
    Let $b$ be a $p$-basis of $F.$ Since $M|F$ is separable and $F(a^{1/p^{\infty}})|F$ is a sub-extension thereof, we get that $b$ is $p$-independent in $M$ and in $F(a^{1/p^{\infty}}).$ By \cite[Lemma 2.2]{at}, $M|F(a^{1/p^{\infty}})$ is separable as long as $b$ is a $p$-basis of $F(a^{1/p^{\infty}}),$ so all we have to prove is that $b$ $p$-spans $F(a^{1/p^{\infty}}),$ i.e., that $$F(a^{1/p^{\infty}})=\left(F(a^{1/p^{\infty}})\right)^{p}(b).$$
    Such equality would follow once we prove that $\left(F(a^{1/p^{\infty}})\right)^{p}=F^{p}(a^{1/p^{\infty}}),$ because $F^{p}(b)=F.$ Now, note that $F^{p}\subseteq\left(F(a^{1/p^{\infty}})\right)^{p}$ and that $a^{1/p^{n}}=(a^{1/p^{n+1}})^{p}\in\left(F(a^{1/p^{\infty}})\right)^{p},$ so $F^{p}(a^{1/p^{\infty}})\subseteq\left(F(a^{1/p^{\infty}})\right)^{p}.$ Conversely, if $x\in\left(F(a^{1/p^{\infty}})\right)^{p},$ then there are some $n<\omega$ and some polynomials $f,g\in F[X_0,\.,X_n]$ such that 
        $$x=\left(\dfrac{f(a,a^{1/p},\.,a^{1/p^{n}})}{g(a,a^{1/p},\.,a^{1/p^{n}})}\right)^{p}=\dfrac{f^{*}(a^{p},a,\.,a^{1/p^{n-1}})}{g^{*}(a^{p},a,\.,a^{1/p^{n-1}})},$$
    where $f^{*}$ and $g^{*}$ are obtained by rising to the $p$-th power the coefficients of $f$ and $g$ respectively. It follows that $x\in F^{p}(a^{1/p^{\infty}}),$ as wanted.
\end{proof}

\begin{lema}
\label{ersgoesup}  
Let $M|F$ be separable algebraic. Then any $p$-basis of $F$ is a $p$-basis of $M$.
\end{lema}
\begin{proof}
    Well known, cf. \cite[Lemma 2.7.3]{fj}.
\end{proof}

\begin{lema}
\label{pindtran}
    Let $M|F$ be a separable field extension, let $\ol{a}$ be a $p$-basis of $M$ over $F$ and let $\ol{c}\subset\ol{a}.$ Then
    \begin{enumerate}[wide, label*={\arabic*.}]
        \item $\ol{c}$ is a $p$-basis of $F(\ol{c})$ over $F,$ and
        
        \item If $a\in\ol{a}\setminus\ol{c},$ then $a$ is not algebraic over $F(\ol{c}).$
    \end{enumerate}
\end{lema}
\begin{proof}
\begin{enumerate}[wide,  label*={\arabic*.}]
    \item First, $\ol{c}$ is $p$-independent in $F(\ol{c})$ over $F.$ Indeed, if there is some $c\in\ol{c}$ such that $c\in (F(\ol{c}))^{p}F(\ol{c}\setminus\{c\})\subseteq M^{p}F(\ol{c}\setminus\{c\}),$ then neither $\ol{c}$ nor $\ol{a}$ are $p$-independent in $M$ over $F.$ Second, $\ol{c}$ is a $p$-spanning tuple of $F(\ol{c})$ over $F,$ because $(F(\ol{c}))^{p}F(\ol{c})=F(\ol{c}),$ as wanted.

    \item Otherwise, the extension $F(\ol{c})(a)|F(\ol{c})$ is algebraic. Since $M|F$ is separable, then so is $F(a,\ol{c})|F.$ It follows that $F(\ol{c})(a)|F(\ol{c})$ would be separable-algebraic. Since $\ol{c}$ is a $p$-basis of $F(\ol{c})$ over $F$ by statement 1, then $\ol{c}$ must be a $p$-basis of $F(a,\ol{c})$ over $F$ too, by Lemma \ref{ersgoesup}. It follows that $a\in (F(a,\ol{c}))^{p}F(\ol{c})=F(a^{p},\ol{c})\subseteq M^{p}F(\ol{c}),$ contradicting $p$-independence of $\{a\}\cup\ol{c}$ in $M$ over $F.$
    \qedhere
\end{enumerate} 
\end{proof}

We finish with a well-known result about so-called \emph{separated} extensions $M|F,$ cf. \cite[Lemma 4.16]{a}.

\begin{lema}
    \label{fincosep}
    Let $M|F$ be a separable extension, and suppose that $M$ and $F$ have the same imperfection degree. If $\i:F\to N$ is a separable embedding and $\widetilde{\i}:M\to N$ is a ring embedding extending $\i,$ then $\widetilde{\i}$ is separable too. 
\end{lema}
\begin{proof}
    We have to prove that $N|\widetilde{\i}M$ is separable. Let $a$ be a $p$-basis of $F.$ Since $M|F$ is separable, $a$ is $p$-independent in $M.$ Since the imperfection degrees of $M$ and of $F$ are equal, we even get that $a$ is a $p$-basis of $M.$ Since $\i$ is separable, we know that $N|\i F$ is separable. Since $\i a$ is $p$-independent in $F,$ then $\i a$ is $p$-independent in $N.$ In order to prove that $N|\widetilde{\i}M$ is separable, we can argue that $\widetilde{\i}a=\i a$ is a $p$-basis of $\widetilde{\i}M$ by, say, \cite[Lemma 2.2]{a}. But this follows from $a$ being a $p$-basis of $M$ and the fact that $\widetilde{\i}$ is a ring embedding. 
\end{proof}

\subsection{Valuation Theory}

We start with some well known lemmas that describe extensions of angular components.

\begin{lema}
\label{uac}
Let $(M,w)\supseteq(F,v)$ be an extension of valued fields. Let $\{a_i:i\in I\}$ be a subset of $M$ such that $wM/vF$ is generated by $\{w(a_i)+vF:i\in I\}.$ If $\ac_1$ and $\ac_2$ are two angular components of $(M,w)$ whose restrictions to $F$ coincide and such that $\ac_1(a_i)=\ac_2(a_i)$ for all $i\in I,$ then $\ac_1=\ac_2.$
\end{lema}
\begin{proof}
    If $x\in M^{\times},$ then $w(x)+vF=w(a_{i_1})+\.+w(a_{i_n})+vF$ for some $i_1,\.,i_n\in I.$ It follows that there is some $b\in F$ such that $$w\left(\dfrac{x}{b\cdot a_{i_1}\cdot\.\cdot a_{i_n}}\right)=0,$$ implying that
    \begin{align*}
    \ac_1(x)&=\ac_1(b)\cdot\ac_1(a_{i_1})\cdot\.\cdot\ac_1(a_{i_n})\cdot\res_w\left(\dfrac{x}{b\cdot a_{i_1}\cdot\.\cdot a_{i_n}}\right)\\
    &=\ac_2(b)\cdot\ac_2(a_{i_1})\cdot\.\cdot\ac_2(a_{i_n})\cdot\res_w\left(\dfrac{x}{b\cdot a_{i_1}\cdot\.\cdot a_{i_n}}\right)\\
    &=\ac_2(x).
    \qedhere
    \end{align*}
\end{proof}

\begin{defin}
    We say that a valued field extension $(M,w)\supseteq(F,v)$ is \emph{unramified} if $wM=vF.$
\end{defin}
In particular, immediate extensions are unramified. This definition usually requires the residue field extension $Mw|Fv$ to be separable, though we will not need this requirement in this note.

\begin{cor}
    Let $(M,w)\supseteq(F,v)$ be an unramified field extension, and let $\ac$ be an angular component of $(F,v).$ Then $\ac$ extends uniquely to an angular component of $(M,w).$
\end{cor}
\begin{proof}
Uniqueness follows from Lemma \ref{uac}. Existence follows by noting that $\ac_w(x):=\ac(b)\cdot\res_w(x/b)$ is a well defined angular component map for $(M,w)$ that extends $\ac,$ where for the given $x\in M^{\times}$ we choose any $b\in F$ such that $w(x)=v(b).$ 
\end{proof}

Now we turn our attention to $\texttt{SAMK}$ and Artin-Schreier closed valued fields.

\begin{lema}
\label{samkasc}
    Let $(M,v)$ be a $\texttt{SAMK}$ valued field. Then $M$ is Artin-Schreier closed.
\end{lema}
\begin{proof}
    Let $N|M$ be an Artin-Schreier extension, i.e. a Galois extension of degree $p.$ Since $(M,v)$ is henselian, the valuation $v$ extends uniquely to $N,$ and for such an extension we have that 
        $$p=[N:M]=p^{d}(vN:vM)[Nv:Mv]$$
    for some $d\in\N.$ If $d=1$ then $(N|M,v)$ is immediate, which is impossible for $(M,v)$ is separable-algebraically maximal. Therefore $d=0$ and $p$ divides either $(vN:vM)$ or $[Nv:Mv],$ which is also impossible because $(M,v)$ is Kaplansky.
\end{proof}

For a given field $M,$ its \emph{Artin-Schreier closure} $M^{\text{AS}}$ is defined to be the algebraic extension of $M$ obtained by successively adjoining Artin-Schreier roots of elements of $M,$ having fixed an algebraic closure $M^{alg}$ of $M.$ This is, if $M_0:=M$ and $$M_{n+1}:=M_n(a\in M^{alg}:a^{p}-a\in M_n),$$ then we may set $M^{\text{AS}}$ to be $\bigcup_{n<\omega}M_n.$ However, if $v$ is a valuation on $M^{alg},$ it will be enough for us to consider the same closure procedure on $M$ but only taking into account Artin-Schreier roots of elements of \emph{negative} valuation. This is, if $M_0:=M$ and $$M_{n+1}:=M_n(a\in M^{alg}:a^{p}-a\in M_n\text{ and }v(a^{p}-a)<0),$$ we define the \emph{(ad-hoc) Artin-Schreier closure} $(M^{as},v)$ of $(M,v)$ as the valued field obtained in $M^{alg}$ by $M^{as}:=\bigcup_{n<\omega}M_n.$ By Lemma \ref{samkasc}, we get that if $(M,v)$ is a $\texttt{SAMK}$ valued field and $F$ is a subfield of $M,$ then $F^{as}\subseteq M.$ 

\begin{lema}
\label{asclosure}
    Let $(F,v)$ be a non-trivial valued field of characteristic $p>0$. 
    \begin{enumerate}[label*={\arabic*.}]
        \item If $F$ closed under Artin-Schreier roots of elements of negative valuation, then $vF$ is $p$-divisible and $Fv$ is perfect.     

        \item $(Fv)^{1/p^{\infty}}\subseteq F^{as}v$ and $\frac{1}{p^{\infty}}vF\subseteq vF^{as}.$ 
    \end{enumerate}
\end{lema}
\begin{proof}
    \begin{enumerate}[wide,  label*={\arabic*.}]
        \item 
        The following proof is the one given in \cite[Lemma 4.4]{fvkpr}, adapted to our context. First, if $c\in F^{\times}$ is such that $v(c)<0$ and $a\in F$ satisfies $a^{p}-a=c,$ then $v(c)/p=v(a)\in vF.$ If $v(c)>0,$ apply the same argument for $v(c^{-1}).$ Second, if $c\in F$ is such that $\res_v(c)\neq 0,$ then $v(c)=0.$ Let $d\in F$ be such that $v(d)<0,$ so that $v(d^{p}c)<0,$ and let $a\in F$ be such that $a^{p}-a=d^{p}c.$ It follows that $v(a)=v(d)>v(d^{p}),$ so applying $\res_v$ to the equation $(a/d)^{p}-(a/d^{p})=c$ yields $\res_v(a/d)^{p}=\res_v(c),$ as wanted.     

        \item Since $(F,v)\subseteq(F^{as},v),$ then $Fv\subseteq F^{as}v$ and $vF\subseteq vF^{as}.$ Hence $(Fv)^{1/p^{\infty}}\subseteq (F^{as}v)^{1/p^{\infty}}=F^{as}v$ and $\frac{1}{p^{\infty}}vF\subseteq\frac{1}{p^{\infty}}vF^{as}=vF^{as},$ where the equalities follow from statement 1.
        \qedhere
    \end{enumerate}
\end{proof}

\begin{lema}
\label{perfcoresat}
    Let $(M,v)$ be an $\aleph_0$-saturated valued field of characteristic $p>0$ and with perfect residue field and $p$-divisible value group.  
    \begin{enumerate}[wide, label*={\arabic*.}]
        \item For each $\a\in Mv$ there is some $a\in\O_v$ such that $\F_p\left(a^{1/p^{\infty}}\right)\subseteq M$ and $\res_v(a^{1/p^{n}})=\a^{1/p^{n}}$ for all $n<\omega.$

        \item For each $\g\in vM$ there is some $a\in M$ such that $\F_p\left(a^{1/p^{\infty}}\right)\subseteq M$ and $v(a^{1/p^{n}})=\g/p^{n}$ for all $n<\omega.$
    \end{enumerate}
\end{lema}
\begin{proof}
    \begin{enumerate}[wide,  label*={\arabic*.}]
        \item Let $\pi(x)$ be the partial type\footnote{In a three sorted language with sorts for the valued field, the residue field and the value group, and where $\underline{v}$ is a symbol for the valuation map and $\underline{\res}$ is a symbol for the residue map.} over $\a$ given by $$\{\underline{v}(x)\geq 0\}\cup\left\{\E y\left(y^{p^{n}}=x \wedge \underline{v}(y)\geq 0 \wedge \underline{\res}(y)^{p^{n}}=\a\right):n<\omega\right\}.$$
        Any realization $a\in M$ of $\pi(x)$ would be the element we are looking for. Since $(M,v)$ is $\aleph_0$-saturated, it is enough to show that $\pi(x)$ is finitely satisfiable. To this end, if $n<\omega,$ then $\a^{1/p^{n}}\in Mv$ because $Mv$ is perfect. Let $c\in\O_v$ be such that $\res_v(c)=\a^{1/p^{n}}.$ Then, for all $i\in\{0,\.,n\},$ we have that $c^{p^{n-i}}\in\O_v$ and $$\res_v\left(c^{p^{n-i}}\right)=\res_v(c)^{p^{n-i}}=\a^{p^{n-i}/p^{n}}=\a^{1/p^{i}},$$ meaning that $c^{p^{n}}$ is a realization of the finite fragment $$\{\underline{v}(x)\geq 0\}\cup\left\{\E y\left(y^{p^{i}}=x \wedge \underline{v}(y)\geq 0 \wedge \underline{\res}(y)^{p^{i}}=\a\right):i\in\{0,\.,n\}\right\}\subseteq\pi(x),$$ as wanted.

        \item Let $\pi(x)$ be the partial type over $\g$ given by $$\left\{\E y\left(y^{p^{n}}=x \wedge p^{n}\underline{v}(y)=\g\right):n<\omega\right\}.$$
        Any realization $a\in M$ of $\pi(x)$ would be the element we are looking for. As above, it is enough to show that $\pi(x)$ is finitely satisfiable. To this end, if $n<\omega,$ then $\g/p^{n}\in vM$ because $vM$ is $p$-divisible. Let $c\in M$ be such that $v(c)=\g/p^{n}.$ Then, for all $i\in\{0,\.,n\},$ we have that  $$v\left(c^{p^{n-i}}\right)=p^{n-i}\cdot v(c)=\dfrac{p^{n-i}}{p^{n}}\g=\g/p^{i},$$ meaning that $c^{p^{n}}$ is a realization of the finite fragment $$\left\{\E y\left(y^{p^{i}}=x\wedge p^{i}\underline{v}(y)=\g\right):i\in\{0,\.,n\}\right\}\subseteq\pi(x),$$ as wanted.
        \qedhere
    \end{enumerate}
\end{proof}

\begin{lema}
\label{nwhit}
    Let $(M,v,\ac)$ be an $\ac$-valued field, and let $F$ be a subfield of $M$ such that $Fv=Mv$.
    \begin{enumerate}[label*={\arabic*.}]
        \item If $\g\in vM,$ then there is some $a\in M$ such that $v(a)=\g$ and $\ac_v(a)=1.$ If $\g\in vF,$ then there is some $c\in F$ such that $v(c)=\g$ and $\ac_v(c)=1.$

        \item Suppose $(M,v)$ is henselian. If $\g\in vM$ and $n\geq1$ is the minimum such that $n\g\in vF,$ and if $n$ and $p$ are coprime, then there is some $a\in M$ such that $v(a)=\g,$ $a^{n}\in F$ and $\ac_v(a)=1.$
    \end{enumerate} 
\end{lema}
\begin{proof}
    \begin{enumerate}[wide,  label*={\arabic*.}]
        \item Let $a\in M$ be such that $v(a)=\g,$ and let $b\in M$ be such that $v(b)=0$ and $\ac_v(b)=\res_v(b)=\ac_v(a).$ Then $\ac_v(a/b)=1$ and $v(a/b)=v(a)=\g.$ Now, if $\g\in vF,$ let $c\in F$ be such that $v(c)=\g.$ Since $\ac_v(c)\in Mv=F_v,$ there is some $d\in F$ such that $v(d)=0$ and $\ac_v(d)=\res_v(d)=\ac_v(c).$ It follows that $c/d\in F,$ that $\ac_v(c/d)=1$ and that $v(c/d)=v(c)=\g,$ as wanted.   

        \item Let $a\in M$ and $c\in F$ be such that $v(a)=\g$ and $v(c)=n\g,$ with $\ac_v(a)=\ac_v(c)=1.$ Then $v(a^{n})=v(c),$ so $v(a^{n}/c)=0$ and $\ac_v(a^{n}/c)=\res_v(a^{n}/c)=1.$ Since $n$ and $p$ are coprime, the polynomial $X^{n}-1$ is separable. By henselianity of $M,$ we find some $d\in M$ such that $v(d-1)>0$ and $d^{n}=a^{n}/c.$ It follows that $v(d)=0$ and $\ac_v(d)=\res_v(d)=1,$ and thus $(a/d)^{n}=c\in F,$ $v(a/d)=v(a)-v(d)=v(a)=\g$ and $\ac_v(a/d)=(\ac_v(a)/\ac_v(d))=1,$ as wanted.
        \qedhere
    \end{enumerate}
\end{proof}

The following result is a slight modification of the corresponding \say{going down} property for tame fields, cf. \cite[Lemma 3.7]{fvktf}. 
\begin{lema}
\label{pop}
    If $(M,v)$ is a \texttt{SAMK}$_p$ valued field and $M|K$ is regular, then $(K,v)$ is separable-algebraically maximal.
\end{lema}
\begin{proof} 
    Let $(F,w)$ be a finite separable-algebraic immediate extension of $(K,v).$ By the fundamental equality, $[F:K]=p^{d}$ for some $d\in\N.$
    Since the extension $M|K$ is regular, it is linearly disjoint from $K^{alg}|K.$
    In particular, $M$ is linearly disjoint from $F$ over $K.$
    Hence, $[MF:M]=p^{d}$ as well.
    Since $F|K$ is separable and $[F:K]=[MF:M],$ the (separable) minimal polynomial over $K$ of any generator $a$ of $F|K$ would be as well the minimal polynomial of the generator $a$ of $MF$ over $M,$ so $MF|F$ is separable.
    Since $MF|M$ is separable-algebraic and $M$ is henselian, there is a unique extension of $v$ to $MF.$ It follows that $p^{d}=p^{e}[MFv:Mv](vMF:vM),$ i.e. $$p^{d-e}=[MFv:Mv](vMF:vM).$$
    If $d-e>0,$ then necessarily $p$ divides either $[MFv:Mv]$ or $(vMF:vM).$ But none of these situations are possible, because $vM$ is $p$-divisible and $Mv$ does not admit any finite extension of degree divisible by $p.$
    It follows that $d=e,$ so the extension $(MF|M,v)$ is immediate and thus trivial, because $M$ is separable-algebraically maximal. 
    This implies that $d=0,$ so $F|K$ is also trivial, as wanted.
\end{proof}

One can do even better for immediate extensions $M|K,$ dropping the assumption of $M|K$ being separable. 
\begin{lema}
\label{popnotsep}
    Let $p$ be a prime number and let $(M|F,v)$ be an immediate extension, where $(F,v)$ is non-trivial, $(M,v)$ is a non-trivial $\texttt{SAMK}_p$ valued field and $F$ is relatively algebraically closed in $M.$ Then $(F,v)$ is a $\texttt{SAMK}_p$ valued field too.
\end{lema}
\begin{proof}
    Since $(M,v)$ is a $\texttt{SAMK}_p$ valued field, it is in particular separably tame, by \cite[Theorem 3.10]{fvktf}. Since the residue field extension $Mv|Fv$ is algebraic (being trivial), and since $F$ is relatively algebraically closed in $M,$ we get that $(F,v)$ is separably tame by \cite[Lemma 3.15]{fvktf} and therefore separable-algebraically maximal by, again, \cite[Theorem 3.10]{fvktf}. 
\end{proof}

\begin{lema}
\label{vtop}
    Let $(M|F,v)$ be a proper valued field extension, and let $B=B_{\geq\g}(b)=\{a\in M:v(a-b)\geq\g\}$ for some $b\in M$ and some $\g\in vM.$ Then $(M\setminus F)\cap B\neq\emptyset.$  
\end{lema}
\begin{proof}
    Otherwise, $B\subseteq F.$ If $a\in M^{\times}$ is such that $\g=v(a),$ then $B=b+a\O_v\subseteq F$ implies that $b$ and $b+a$ are elements of $F,$ so $a\in F-F\subseteq F,$ and thus $$\O_v=a^{-1}(B-b)\subseteq F^{-1}(F-F)\subseteq F.$$ Finally, $M=\Frac\O_v\subseteq \Frac F=F,$ so $M|F$ is not proper.   
\end{proof}

\subsection{Kaplansky Theory}

We recall some fundamental facts about pseudo-Cauchy sequences, their algebraic or transcendental character, and the role of Kaplansky's hypotheses in our setting. Recall that if $(K,v)$ is a valued field and $\{a_\nu\}_\nu$ is a pseudo-Cauchy sequence, we say that it is \emph{of algebraic type over $K$} if there is some polynomial $P(X)\in K[X]$ such that $\{P(a_\nu)\}_\nu$ pseudo-converges to $0.$ If this is not the case, we call $\{a_\nu\}_\nu$ \emph{of transcendental type over $K$}. If $\{a_\nu\}_\nu$ is of algebraic type over $K,$ then \emph{a minimal polynomial of $\{a_\nu\}_\nu$ over $K$} is a polynomial $P(X)\in K[X]$ of minimal degree such that $\{P(a_\nu)\}_\nu$ pseudo-converges to $0.$ In general, minimal polynomials of pseudo-Cauchy sequences are \emph{not} unique, even by restricting to monic polynomials. 

\begin{fact}[Cf. {\cite[Theorems 2 and 3]{kap}} or {\cite[Theorems 4.9 and 4.10]{vdd}}]
\label{tralgtype}
Let $(K,v)$ be a valued field.
    \begin{enumerate}[label*={\arabic*.}]
        \item Let $\{a_\nu\}_\nu$ be a pseudo-Cauchy sequence in $K$ of transcendental type over $K$. Then $\{a_\nu\}_\nu$ has no pseudo-limit in $K^{alg}$. The valuation $v$ on $K$ extends uniquely to a valuation $v:K(X)\to vK\cup\{\infty\}$ such that
        $$v(f(x))=\text{ eventual value of }v(f(a_\nu))$$
        for each $f\in K[X].$ With this valuation, $(K(X)|K,v)$ is an immediate valued field extension in which $X$ is a pseudo-limit of $\{a_\nu\}_\nu$. Conversely, if $(L|K,v)$ is an extension and $a\in L$ is a pseudo-limit of $\{a_\nu\}_\nu,$ then there is a
        valued field isomorphism $K(X)\to K(a)$ over $K$ that sends $X$ to $a$.
        
        \item Let $\{a_\nu\}_\nu$ be a pseudo-Cauchy sequence in $K$ of algebraic type over $K$, without
        pseudo-limit in $K$, and let $\mu\in K[X]$ be a minimal polynomial of $\{a_\nu\}_\nu$ over $K.$ Then $\mu$ is
        irreducible in $K[X]$ and $\deg\mu\geq 2.$ Let $a$ be a zero of $\mu$ in a field extension of $K$.
        Then $v$ extends uniquely to a valuation $v:K(a)\to vK\cup\{\infty\}$ such that
        $$v(f(a))=\text{ eventual value of }v(f(a_\nu))$$
        for each non-zero polynomial $f\in K[x]$ of degree $\deg f<\deg\mu.$ With this valuation, $(K(a)|K,v)$
        is an immediate valued field extension and $a$ is a pseudo-limit of $\{a_\nu\}_\nu.$
        Conversely, if $\mu(b)=0$ and $b$ is a pseudo-limit of $\{a_\nu\}_\nu$ in a valued field extension of $(K,v)$ then
        there is a valued field isomorphism $K(a)\to K(b)$ over $K$ sending $a$ to $b.$ 
    \end{enumerate}
\end{fact}

The following is the most fundamental result about Kaplansky fields that we are going to use. It will play a central role in the proof of our Embedding Lemma \ref{emblem}.

\begin{fact}[Cf. {\cite[Theorem 5]{kap}}]
    \label{kap}
    Let $(F,v)$ be a Kaplansky field. Then any two maximal immediate extensions of $(F,v)$ are isomorphic over $(F,v).$ 
\end{fact}

Let $(M|F,v)$ be an extension of valued fields where $(M,v)$ is complete. We will use the notation $(F^{c},v)$ to mean the completion of $(F,v),$ this is, if $a\in M,$ then $a\in F^{c}$ if and only if there is some Cauchy sequence $\{a_\nu\}_\nu$ in $F$ that converges to $a.$  

\begin{fact}[Cf. {\cite[Proposition 3.11]{fvkp}}]
\label{sampc}
    Let $(M|F,v)$ be an immediate extension where $(F,v)$ is a separable-algebraically maximal valued field. Suppose there is some $a\in M$ together with some pseudo-Cauchy sequence $\{a_\nu\}_\nu$ in $F$ of algebraic type over $F$ and without pseudo-limit in $F,$ such that $\{a_\nu\}_\nu$ pseudo-converges to $a.$ Then $a\in F^{c}.$ 
\end{fact}

In fact, the hypothesis that says that $\{a_\nu\}_\nu$ does not admit any pseudo-limit in $F$ is not included in \cite[Proposition 3.11]{fvkp}, and it turns out to be a necessary hypothesis that we will check in further applications of this fact. However, their proof works just as well under this hypothesis. We are grateful to Franz-Viktor Kuhlmann for pointing this issue out.

\begin{remark}
\label{trt}
    Under the hypotheses of Fact \ref{sampc}, if $a$ is transcendental over $F$ and $\{b_\nu\}_\nu$ is a Cauchy sequence in $F$ that converges to $a,$ then necessarily $\{b_\nu\}_\nu$ is of transcendental type. Otherwise, if it is of algebraic type and $\mu\in F[X]$ is a minimal polynomial thereof, by continuity of polynomials, we would get that $\{\mu(b_\nu)\}_\nu$ is a Cauchy sequence that converges both to $0$ and to $\mu(a),$ yielding that $\mu(a)=0$ by uniqueness of limits of Cauchy sequences. 
\end{remark}

\subsection{Model Theory}

\subsubsection{Existential Fragments}

The following notations are taken from Chapter 3 of \cite{tz}.

\begin{defin}
    Let $\LL$ be a language, $\D$ be a subset of $\LL$-formulas, and let $\AA,\BB$ be two $\LL$-structures. We write $\AA\Rightarrow_\D\BB$ if $\AA\models\f$ implies that $\BB\models\f$ for every sentence $\f\in\D,$ and if $f:A\to B$ is a function, we write $f:\AA\to_\D\BB$ if $\AA\models\f(a)$ implies that $\BB\models\f(f(a))$ for all formulas $\f(x)\in\D$ and all tuples $a$ from $A,$ i.e. if $(\AA,a)_{a\in A}\Rightarrow_{\D}(\BB,f(a))_{a\in A}.$ We write $\AA\equiv_\D\BB$ if $\AA\Rightarrow_{\D}\BB$ and $\BB\Rightarrow_{\D}\AA.$ Finally, if $\AA\subseteq\BB,$ we write $\AA\preceq_\D\BB$ if $(\AA,a)_{a\in A}\equiv_\D(\BB,a)_{a\in A}.$ 
\end{defin}

The following fact is well known. The referenced result only shows that statement 1 is equivalent to statement 3, but statement 2 follows from the proof (without explicit mention of this in the reference).

\begin{fact}[Cf. {\cite[Lemma 3.1.2]{tz}}]
\label{exemb}
Let $\AA$ and $\BB$ be two $\LL$-structures, and let $\D=\E$ be the set of existential $\LL$-formulas. The following statements are equivalent.
\begin{enumerate}[wide, label*={\arabic*.}]
    \item $\AA\Rightarrow_\E\BB.$

    \item For all $|A|^{+}$-saturated elementary extensions $\BB^{*}$ of $\BB,$ there is some function $f:\AA\to_\E\BB^{*}.$

    \item There is some $\BB^{*}\succeq\BB$ and some function $f:\AA\to_\E\BB^{*}.$ 
\end{enumerate}
\end{fact}

For any subset $\D$ of $\LL$-sentences, and any $\LL$-theory $T,$ consider the space $S^{\D}(T)$ of complete $\D$-types. If $\Phi\in\D,$ define the subset $D(\Phi)=\{p\in S^{\D}(T):\Phi\not\in p\}.$ The family of subsets $\{D(\Phi):\Phi\in\D\}$ determines a sub-basis of a topology on $S^{\D}(T).$ 
The following lemma should be well known, as it is an application of the fact that $S^{\D}(T)$ is a spectral topological space, cf. \cite[Theorem 14.2.5]{ss}.

\begin{fact}[Cf. {\cite[Theorem 14.2.17]{ss}}] 
\label{delredeqf}
Let $T$ be an $\LL$-theory and let $\D$ be a set of $\LL$-formulas closed under conjunctions and disjunctions. Let $\Phi,\Psi$ be subsets of $\LL$-formulas. The following are equivalent.
\begin{enumerate}[wide, label*={\arabic*.}]
    \item There is a set $\S\subseteq\D\cup\{\bot\}$ such that $T\cup\Phi\vdash\Sigma$ and $T\cup\Sigma\vdash\Psi,$

    \item If $p,q\in S(T)$ and $p\cap\D\subseteq q,$ then $\Phi\subseteq p$ implies $\Psi\subseteq q.$
\end{enumerate}
If $\Psi$ is finite, $\Sigma$ can be chosen to be a single formula.
\end{fact}

\begin{cor}
\label{delredeq}
Let $T$ be an $\LL$-theory and let $\D'$ be a set of $\LL$-formulas closed under conjunctions and disjunctions. 
\begin{enumerate}[label={\arabic*.}]
    \item Let $\f(x)$ be an $\LL$-formula. The following statements are equivalent.
        \begin{enumerate}
            \item There is a formula $\psi(x)\in\D'\cup\{\bot\}$ such that $T\vdash\A x(\f(x)\sii\psi(x)),$
        
            \item For any pair of types $p,q\in S_x(T),$ if $p\cap\D'\subseteq q$ and $\f(x)\in p,$ then $\f(x)\in q.$
        \end{enumerate}
    \item Let $\D$ be a set of $\LL$-formulas. The following statements are equivalent.
        \begin{enumerate}
            \item Any formula from $\D$ is equivalent, modulo $T,$ to a formula from $\D'\cup\{\bot\}.$

            \item For any finite tuple of variables $x$ and any pair of types $p,q\in S_x(T),$ if $p\cap\D'\subseteq q,$ then $p\cap\D\subseteq q.$
            
            \item For any pair of models $\MM,\NN$ of $T$ and any pair of tuples $a\in M,b\in N$ of the same length, if $(\MM,a)\Rightarrow_{\D'}(\NN,b),$ then $(\MM,a)\Rightarrow_{\D}(\NN,b).$
        \end{enumerate}
\end{enumerate}
\end{cor}
\begin{proof}
Statement 1 follows from Fact \ref{delredeqf}. The equivalence between (a) and (b) of statement 2 follows from statement 1. It is clear that (a) implies (c), so we just have to prove that (c) implies (b). To this end, let $x$ be a finite tuple of variables, and let $p,q\in S_x(T)$ be two types. Then there are two models $\MM,\NN$ of $T$ and two $|x|$-tuples $a,b$ of $M$ and $N$ respectively such that $p=\tp^{\MM}(a)$ and $q=\tp^{\NN}(b).$ If $p\cap\D'\subseteq q,$ then $(\MM,a)\Rightarrow_{\D'}(\NN,b),$ which by hypothesis implies that $(\MM,a)\Rightarrow_{\D}(\NN,b),$ meaning that $p\cap\D\subseteq q,$ as wanted. 
\end{proof}

\subsubsection{Languages and Theories of \texttt{SAMK} Valued Fields}

For the rest of this document, we will work with three sorted languages for valued fields, with sorts $\KK$ for the home field, $\kk$ for the residue field and $\GG$ for the value group. If $\LL$ is any such language, we write $\LL(\KK)$ (resp. $\LL(\kk),$ $\LL(\GG)$) to mean the language associated to the sort $\KK$ (resp. $\kk$ and $\GG$), and we always assume that any such $\LL$ contains a function symbol $\underline{v}:\KK\to\GG$ to be interpreted as a valuation, a function symbol $\underline{\res}:\KK\to\kk$ to be interpreted as the associated residue map, and we also assume that $\LL(\KK)$ and $\LL(\kk)$ contain the language of rings $\LL_{ring}$ and $\LL(\GG)$ contains the language of ordered groups $\LL_{og}$ together with a symbol $\infty$ for a point at infinity. In order to stress the dependence on the auxiliary languages $\LL(\kk)$ and $\LL(\GG),$ we will even write $\LL=\LL(\LL_\kk,\,\LL_\GG)$ to mean that $\LL(\kk)=\LL_\kk$ and $\LL(\GG)=\LL_\GG$ for a given couple of languages $\LL_\kk$ and $\LL_\GG$ containing $\LL_{ring}$ and $\LL_{og}\cup\{\infty\}$ respectively. 
Let $\LL_{3s}=\LL_{3s}(\LL_\kk,\,\LL_\GG)$ be the three sorted language for valued fields, \emph{possibly expanded in $\LL_{3s}(\GG)$ and $\LL_{3s}(\kk),$} i.e. where $\LL_\kk$ and $\LL_\GG$ are allowed to have more symbols than $\LL_{ring}$ and $\LL_{og}\cup\{\infty\}$ respectively. Note that the class of $\texttt{SAMK}$ equi-characteristic valued fields is elementary in any such $\LL_{3s},$ axiomatized by the following axioms:
\begin{itemize}[wide]
    \item The axioms of henselian valued fields,

    \item The axioms of extremality for separable polynomials in one variable,

    \item Equi-characteristic: $$\left\{\sum_{i=1}^{p}1_{\KK}=0\sii\sum_{i=1}^{p}1_{\kk}=0:p\text{ prime}\right\},$$
    where $1_\KK$ is the unit symbol of the home field sort $\KK$ and $1_\kk$ is the unit of the residue field sort $\kk,$
    \item Kaplansky (uniform across all positive characteristics): $$\left\{\sum_{i=1}^{p}1_{\kk}=0\to\A x\E y\left(v(x)=pv(y)\right):p\text{ prime}\right\},$$ and $$\left\{\sum_{i=1}^{p}1_{\kk}=0\to\A x_0,\.,x_n,y\,\E z\left(\sum_{i=0}^{n}x_iz^{p^{i}}=y\right):p\text{ prime}, n<\omega\right\},$$ where all the mentioned variables are of sort $\kk.$
\end{itemize}

Define $\LL_{\ac}$ as the expansion of $\LL_{3s}$ by a function symbol for an angular component $\underline{\ac}:\KK\to\kk$, and let $\LL$ be the expansion of $\LL_{\ac}$ by countably many function symbols $\{\underline{\l}_{n,m}(x,y):|x|=n\geq1,m\in\Mon(n)\}$ where each $\underline{\l}_{n,m}$ is of sort $\KK^{n+1}\to\KK.$ An $\LL_{3s}$-structure whose underlying universe is an $\ac$-valued field $(M,v,\ac_v)$ of characteristic $p>0$ is then an $\LL$-structure $\MM$ via the following interpretation:  
$$\underline{\l}_{n,m}^{\MM}(a,b)=\begin{cases}
    \l_{m}^{a}(b)&\text{ if }a\text{ is a }p\text{-independent }n\text{-tuple in }M\text{ and }b\in M^{p}(a),\\
    0&\text{ otherwise.}
\end{cases}$$
For $e\in\N\cup\{\infty\},$ let $\texttt{SAMK}_{e}^{\l,\ac}$ be the $\LL$-theory given by the following axioms:
\begin{itemize}[wide]
    \item The $\LL_{\ac}$-theory of non-trivial $\texttt{SAMK}$ equi-characteristic valued fields admitting an angular component,

    \item If $e\in\N,$ the $\LL_{ring}$-axioms fixing the Ershov degree of the home sort equal to $e:$ If $x_1,\.,x_n$ are variables from the sort $\KK,$ let $\Mon(x_1,\.,x_n)=\{\prod_{i=1}^{n}x_i^{\a_i}:0\leq\a_i<p\text{ for all }i\in\{1,\.,n\}\}$ be the set of monomials obtained with these variables. Then the axioms (one for each prime $p$) say that if $\sum_{i=1}^{p}1_\KK=0,$ then there are $e$ elements $x_1,\.,x_e$ whose set of monomials $\Mon(x_1,\.,x_e)$ is linearly independent over $\KK^{p}$ and generates $\KK$ over $\KK^{p}.$

    \item If $e=\infty,$ the axioms fixing the Ershov degree of the home sort equal to $e:$ with the same notation as above, the axioms (one for each $e\in\N$ and each prime $p$) say that if $\sum_{i=1}^{p}1_\KK=0,$ then there are $e$ elements $x_1,\.,x_e$ whose set of monomials $\Mon(x_1,\.,x_e)$ is linearly independent over $\KK^{p}.$  

    \item The axioms describing the $\l$-maps: for any $n<\omega$ and any prime $p,$ if $\sum_{i=1}^{p}1_\KK=0,$ then $$t=\sum_{m\in\Mon(n)}(\l_{m}^{x}(t))^{p}\cdot x^{m}$$ for all tuples $x=(x_1,\.,x_n)$ whose monomial set $\Mon(x_1,\.,x_n)$ is linearly independent over $\KK^{p}$ and for all $t$ in the $\KK^{p}$-span of $\Mon(x_1,\.,x_n),$ cf. Remark \ref{defl}.
\end{itemize}

In our setting, if $\LL=\LL(\LL_\kk,\,\LL_\GG),$ an $\LL$-structure $\AA$ is a tuple $$\left(A,\kk(A),\GG(A),v,\res_v,\ac_v,\l_{n,m}^{A}:n\in\N^{>0},m\in\Mon(n)\right)$$ where
\begin{itemize}[wide]
    \item $A$ and $\kk(A)$ are rings and $\GG(A)$ is an ordered group, where $\kk(A)$ and $\GG(A)$ may admit further structure coded by $\LL_\kk$ and $\LL_\GG$ respectively, and

    \item $\begin{cases}
     v:A\to\GG(A)\cup\{\infty\},\\
     \res_v,\ac_v:A\to\kk(A),\\
     \l_{n,m}^{A}:A^{n+1}\to A
    \end{cases}$ are functions. 
\end{itemize}
If $\AA$ and $\BB$ are $\LL$-structures as above (with symbols $v$ and $w$ respectively), then an $\LL$-embedding $\Sigma:\AA\to\BB$ is a triple $(\iota,\sigma,\rho)$ where 
\begin{itemize}[wide]
    \item $\iota:A\to B$ is a ring embedding,

    \item $\sigma:\kk(A)\to\kk(B)$ is an $\LL_\kk\,$-embedding and $\rho:\GG(A)\cup\{\infty\}\to\GG(B)\cup\{\infty\}$ is an $\LL_\GG\,$-embedding. Note that $\LL_\kk\supseteq \LL_{ring}$ and $\LL_\GG\supseteq\LL_{og}\cup\{\infty\},$ so, in particular, $\sigma$ is a ring embedding and $\rho$ is a strictly increasing group homomorphism respecting the laws for $\infty;$ and

    \item the squares 
    \begin{center}
        \begin{tikzcd}
A \arrow[r, "v"] \arrow[d, "\iota"'] & \GG(A)\cup\{\infty\} \arrow[d, "\rho"] & A \arrow[d, "\iota"'] \arrow[rr, "{\res_v,\,\,\ac_v}"] &  & \kk(A) \arrow[d, "\sigma"] & A^{n+1} \arrow[d, "\iota"'] \arrow[r, "{\l_{n,m}^{A}}"] & A \arrow[d, "\iota"] \\
B \arrow[r, "w"]                     & \GG(B)\cup\{\infty\}                   & B \arrow[rr, "{\res_w,\,\,\ac_w}"]                     &  & \kk(B)                     & B^{n+1} \arrow[r, "{\l_{n,m}^{B}}"]                     & B                   
\end{tikzcd} 
    \end{center}
    commute.   
\end{itemize}

In particular, if $\MM$ is an $\LL$-structure whose $\KK$-sort universe $M$ is a field, and $F$ is a $\l$-closed subfield of $M,$ then the structure generated by $F,$ denoted by $\langle F\rangle,$ corresponds to the tuple $$\left(F,Fv,vF,v|_F,\res_v|_{F},\ac_v|_{F},\l_{n,m}^{M}|_{F}:n\in\N^{>0},m\in\Mon(n)\right),$$ i.e. $\GG(F)=vF$ and $\kk(F)=Fv.$ In this case, $(F,v|_F,\ac_v|_{F})$ is a well-defined $\ac$-valued field.

The following Lemma allows us to construct models of $\texttt{SAMK}_e^{\l,\ac}$ for preferred Ershov degrees $e$, $p$-closed residue fields and $p$-divisible value groups. 

\begin{lema}
\label{cons}
    Let $\LL=\LL(\LL_\kk,\,\LL_\GG),$ and consider the $\LL$-theory $\mathcal{T}(e,T_{\kk},T_{\GG})=\texttt{SAMK}_{e}^{\l,\ac}\cup T_{\kk}\cup T_{\GG},$ where 
    \begin{itemize}[wide]
        \item $e\in\N\cup\{\infty\},$
    
        \item $T_{\kk}$ is a consistent $\LL_\kk\,$-theory of fields such that $$T_{\kk}\vdash\left\{\sum_{i=1}^{p}1_{\kk}=0\to\A x_0,\.,x_n,y\,\E z\left(\sum_{i=0}^{n}x_iz^{p^{i}}=y\right):p\text{ prime}, n<\omega\right\},$$

        \item $T_{\GG}$ is a consistent $\LL_\GG\,$-theory of non-trivial ordered abelian groups,
    \end{itemize}
    and allow these theories to admit extra axioms for the symbols of $\LL_\kk\setminus\LL_{ring}$ and $\LL_\GG\setminus\LL_{og}\cup\{\infty\}$ respectively. Let $P_1$ be the set of all possible characteristics of models of $T_{\kk},$ and let $P_2$ be the set of all prime numbers $p$ such that $T_{\GG}$ admits a $p$-divisible model. Then for any Ershov degree $e\in\N\cup\{\infty\},$ the theory $\mathcal{T}(e,T_{\kk},T_{\GG})$ is consistent if and only if $P_1\cap P_2\neq\emptyset.$
\end{lema}
\begin{proof}
     The direct implication is clear by the axioms of Kaplansky fields. For the converse implication, let $k\models T_{\kk}$ and $\G\models T_{\GG}$ be countable models corresponding to a common $p\in P_1\cap P_2.$ If $F:=k(t^{\g}:\g\in\G),$ then $\left(\hs{k}{\G}\,|F,v_t\right)$ is an immediate extension of Kaplansky fields, where $v_t$ is the usual $t$-adic valuation for Hahn fields, and $F$ is perfect. If $\ac_t$ is the usual angular component map of $\hs{k}{\G},$ then its restriction to $F$ is still surjective, i.e. $(F,v_t|_{F},\ac_t|_{F})$ is a bona fide $\ac$-valued field.
     Since $|F|=\aleph_0<2^{\aleph_0}=|\hs{k}{\G}|,$ there is a sequence $(s_n)_{n<\omega}$ of elements of $\hs{k}{\G}$ which is algebraically independent over $F.$ Then $F(s_n)_{n<e}$ has Ershov degree $e$ if $e\in\N,$ or $\infty$ if $e=\omega,$ by \cite[Lemma 2.7.2]{fj}. Since $F(s_n)_{n<e}|F$ is a subextension of the immediate extension $\hs{k}{\G}\,|F,$ it is in particular Kaplansky. Therefore, if $v_0$ is the restriction of $v_t$ to $F(s_n)_{n<e},$ and if $v$ is any valuation on $F(s_n)_{n<e}^{alg}$ extending $v_0,$ then by Kaplansky's Uniqueness Theorem \cite[Theorem 5]{kap}, $F(s_n)_{n<e}$ admits a unique maximal immediate algebraic extension $(K|F(s_n)_{n<e},v)$ up to isomorphism over $F(s_n)_{n<e},$ making $(F(s_n)_{n<e}^{sep}\cap K,v)$ a separable-algebraically maximal Kaplansky field of imperfection degree equal to $e.$ Indeed, if $L$ is a separable-algebraic immediate extension of $K\cap F(s_n)_{n<e}^{sep},$ then $L$ can be embedded into some maximal immediate algebraic extension $K'$ which would still be contained in $F(s_n)_{n<e}^{alg},$ so there is an embedding $K'\hookrightarrow K$ over $F(s_n)_{n<e}$ by uniqueness of $K,$ and $L\subseteq F(s_n)_{n<e}^{sep}$ by assumption, implying that $L=F(s_n)_{n<e}^{sep}\cap K$. If $\ac_0$ is the restriction of $\ac_t$ to $F(s_n)_{n<e},$ and since the extension $(F(s_n)_{n<e}^{sep}\cap K|F(s_n)_{n<e}^{sep},v)$ is in particular unramified, by Lemma \ref{uac} we can define a (uniquely determined) angular component $\ac$ on $F(s_n)_{n<e}^{sep}\cap K$ extending $\ac_0.$ In other words, the structure $(F(s_n)_{n<e}^{sep}\cap K,v,\ac)$ is a model of $\mathcal{T}(e,T_\kk,T_{\GG})$ when it is endowed with its own parameterized $\l$-maps.     
\end{proof}

\begin{center}
\begin{tikzcd}
                                              &                                        & F(s_n)_{n<e}^{alg}                                                                   &                       &                                \\
                                              & F(s_n)_{n<e}^{sep} \arrow[ru, no head] &                                                                                      & K \arrow[lu, no head] &                                \\
{\texttt{SAMK}_e^{\l,\ac}} \arrow[rr, dotted] &                                        & F(s_n)_{n<e}^{sep}\cap K \arrow[ru, no head] \arrow[lu, no head] \arrow[rd, no head] &                       & \hs{k}{\G} \arrow[ld, no head] \\
                                              &                                        &                                                                                      & F(s_n)_{n<e}          &                                \\
                                              &                                        &                                                                                      & F \arrow[u, no head]  &                               
\end{tikzcd}
\end{center}

\begin{cor}
\label{disj}
Let $p$ be a prime number, let $T_{\kk}$ be a consistent theory of $p$-closed fields of characteristic $p$, let $T_{\GG}$ be a theory of $p$-divisible ordered abelian groups and let $T=\mathcal{T}(e,T_{\kk},T_{\GG})$ be as in Lemma \ref{cons}, for a fix Ershov degree $e\in\N\cup\{\infty\}.$ Let $\th$ and $\psi$ be $\LL_\kk\,$- and $\LL_\GG\,$-sentences respectively.
\begin{enumerate}[wide, label*={\arabic*.}]
    \item $T\vdash\th\vee\psi$ if and only if $T_{\kk}\vdash\th$ or $T_{\GG}\vdash\psi.$

    \item If $T_\kk$ and $T_\GG$ are complete and $\th$ and $\psi$ are existential, then $T\vdash\th\vee\psi$ if and only if $(T_{\kk})_\E\vdash\th$ or $(T_{\GG})_\E\vdash\psi.$
\end{enumerate}
\end{cor}
\begin{proof}
    \begin{enumerate}[wide, label*={\arabic*.}]
        \item The reverse implication follows from $T\vdash T_{\kk}\cup T_{\GG}$. For the direct implication, assume that $T_1=T_{\kk}\cup\{\neg\th\}$ and $T_2=T_{\GG}\cup\{\neg\psi\}$ are consistent. Using the notation of Lemma \ref{cons}, we have by assumption that $p\in P_1\cap P_2,$ implying that the $\LL$-theory $T^{*}=\mathcal{T}(e,T_1,T_2)$ is consistent. But $T^{*}\vdash T$ and $T^{*}\vdash\neg\th\wedge\neg\phi,$ contradicting that $T\vdash\th\vee\psi.$  

        \item The reverse inclusion follows from $T\vdash (T_{\kk})_\E\cup (T_{\GG})_\E.$ For the direct implication, suppose that $T_\kk=\Th_{\LL_\kk}(F)$ and $T_\GG=\Th_{\LL_\GG}(G)$ for some $p$-closed field $F$ and some $p$-divisible ordered abelian group $G.$ If $(T_\kk)_\E\not\vdash\th$ and $(T_\GG)_\E\not\vdash\psi,$ then there is some field $k$ and some group $\G$ such that $k\models \Th_\E(F)\cup\{\neg\th\}$ and $\G\models \Th_\E(G)\cup\{\neg\psi\}.$ It follows that $F\Rightarrow_\E k$ and $G\Rightarrow_\E \G,$ and since $\th$ and $\psi$ are existential, we have that $F\models\neg\th$ and $G\models\neg\psi.$ Therefore $T^{*}=\mathcal{T}(e,\Th_{\LL_\kk}(F),\Th_{\LL_\GG}(G))$ is consistent and $T^{*}\vdash\neg\th\wedge\neg\psi.$ As above, this contradicts the fact that $T^{*}\vdash T$ and that $T\vdash\th\vee\psi.$          
        \qedhere
    \end{enumerate}
\end{proof}

\subsection{Computability}

Let $\LL$ be a \emph{countable} language, let $T$ be an $\LL$-theory and let $\D'\subseteq\D$ be two sets of $\LL$-formulas. 
If we fix an injection $\a:\LL\to\N,$ by a standard Gödel coding we obtain an injection $\widetilde{\a}:\text{Form}(\LL)\to\N$ defined on the set $\text{Form}(\LL)$ of $\LL$-formulas. If $T_1$ and $T_2$ are two $\LL$-theories, we say that $T_1$ is \emph{many-one reducible} to $T_2$ if there is a computable function $f:\N\to\N$ such that $n\in\widetilde{\a}(T_1)$ if and only if $f(n)\in\widetilde{\a}(T_2).$ We denote this by $T_1\leq_m T_2.$ We say that $T_1$ is \emph{Turing reducible} to $T_2,$ denoted by $T_1\leq_T T_2,$ if $\widetilde{\a}(T_1)\leq_T\widetilde{\a}(T_2),$ i.e. if there is a Turing machine with an oracle for $\widetilde{\a}(T_2)$ that computes the characteristic function of $\widetilde{\a}(T_1).$ We will also denote by $T_1\oplus T_2$ the set $2\widetilde{\a}(T_1)\cup(2\widetilde{\a}(T_2)+1)\subseteq\N.$ 
The following lemma is well known.
\begin{lema}
    Suppose that $T,$ $\D'$ and $\D$ are computably enumerable. If every $\D$-formula is equivalent modulo $T$ to a $\D'$-formula, then there is a computable function $\t:\D\to\D'$ such that $T\vdash\A x(\f(x)\sii\t\f(x))$ for all $\f(x)\in\D.$
\end{lema}
\begin{proof}[Proof Sketch]
    Suppose $\{\psi_n:n\in\N\}$ is a computable enumeration of $\D',$ and for a given $\f\in\D$, a given $\psi\in\D'$ and a given length $m\in\N,$ let $\{P_{m,n}(\f,\psi):n\in\N\}$ be a computable enumeration of all possible proofs modulo $T$ of length $m$ of the sentence $\A x(\f(x)\sii\psi(x)).$ This computable enumeration exists because $T$ is computably enumerable. Finally, let $f(n)=(f_1(n),f_2(n),f_3(n))$ be a computable enumeration of $\N^{3}.$ The function $\t$ runs the following algorithm: with input $\f(x)\in\D,$ let $n=0.$ While $n\geq0,$ print $\psi_{f_1(n)}$ and check wether $P_{f_2(n),f_3(n)}(\f(x),\psi_{f_1(n)})$ is a proof. If so, return $\t\f(x)=\psi_{f_1(n)}(x),$ else let $n=n+1.$ Since every $\D$-formula is equivalent modulo $T$ to a $\D'$-formula, there has to be some $n\in\N$ such that $P_{f_2(n),f_3(n)}(\f(x),\psi_{f_1(n)}(x))$ is a proof, meaning that the algorithm halts at every input $\f(x)\in\D$. 
\end{proof}

For an $\LL$-theory $T$ and a set of $\LL$-formulas $\D,$ we will write $T_\D$ for the set $\{\f\in\D:T\vdash\f\}.$ 
Using the terminology of Anscombe and Fehm \cite{af}, the function $\t$ described in the latter Lemma defines a computable \emph{translation} from the context $(\D,T)$ to the context $(\D',T).$ 
The following fact will allow us to conclude moreover that, under the same hypotheses, $T_{\D}$ is many-one reducible to $T_{\D'}.$ 

\begin{fact}[Cf. {\cite[Lemma 2.5]{af}}]
    Let $\t$ be a translation from the context $(\D,T)$ to the context $(\D',T).$ Then $T_\D=\t^{-1}(T_{\D'}).$ If moreover $\D,\D'$ and $\t$ are computable, then $T_\D\leq_mT_{\D'}.$
\end{fact}

\begin{cor}
\label{mto}
    Suppose that $T,$ $\D'$ and $\D$ are computably enumerable. If every $\D$-formula is equivalent modulo $T$ to a $\D'$-formula, then $T_\D\leq_mT_{\D'}.$
\end{cor}

\section{The Embedding Lemma}

The following is the key result of this document, and is inspired by the Separable Relative Embedding Property of Separably-Tame valued fields, cf. \cite[Lemma 4.3]{fvkp}. Recall that $\LL_\kk$ and $\LL_\GG$ are allowed to have more symbols than $\LL_{ring}$ and $\LL_{og}\cup\{\infty\}$ respectively. 

\begin{teorema}[Embedding Lemma]
\label{emblem}
    Let $\LL=\LL(\LL_\kk,\,\LL_\GG),$ let $\MM,\NN$ be two models of $\texttt{SAMK}_{e}^{\l,\ac}$ and let $\AA$ be the $\LL$-substructure of $\MM$ generated by a $\l$-closed subring $A$ of $M.$ Suppose that $\MM$ is $\aleph_0$-saturated and that $\NN$ is $|M|^{+}$-saturated. 
    Let $\sigma:\kk(M)\to\kk(N)$ and $\rho:\GG(M)\to\GG(N)$ be some $\LL_\kk\,$- and $\LL_\GG\,$-embeddings respectively. 
    If $\iota:A\to N$ is a ring embedding making $\Sigma=(\iota,\sigma|_{Av},\rho|_{vA}):\AA\to\NN$ into an $\LL$-embedding, then there is some ring embedding $\widetilde{\iota}:M\to N$ such that $\widetilde{\Sigma}:=(\widetilde{\iota},\sigma,\rho):\MM\to\NN$ is an $\LL$-embedding extending $\Sigma.$ 
\end{teorema}

For the proof of this theorem, the driving idea is to split the ring extension $A\subseteq M$ into smaller sub-extensions $$A\subseteq F_0\subseteq F_1\subseteq\.\subseteq F_7\subseteq M,$$ where $F_0,\.,F_7$ are subfields of $M,$ and argue that $\iota$ can be extended each time from $A$ to $F_0,$ from $F_i$ to $F_{i+1}$ for $i\in\{0,\.,6\},$ and from $F_7$ to $M.$ At each stage, we will make sure that the extension $\widetilde{\iota}$ of $\iota$ induces an $\LL$-embedding $\widetilde{\Sigma}=(\widetilde{\iota},\sigma,\rho)$ extending the $\Sigma$ of the step before. This choice of presentation is influenced by Hils' exposition of Pas' field quantifier elimination \cite{mh}.

\begin{proof}[Proof of Theorem \ref{emblem}]
 
Let $p$ be the characteristic exponent of $A,$ i.e. $p=1$ if $\char(A)=0$ and $p=\char(A)$ otherwise.
    \begin{enumerate}[wide, label=\underline{Step \arabic*}]{}
    \item: \emph{$\iota$ extends to $F_0=\Frac(A).$} Indeed, if $a,b\in A$ are non-zero, we define the extension $\widetilde{\iota}(a/b):=\iota(a)/\iota(b).$ Then 
    \begin{align*}
        \rho(v(a/b))&=\rho(v(a)-v(b)) & \sigma(\ac_v(a/b))&=\sigma(\ac_v(a)/\ac_v(b))\\
        &=\rho(v(a))-\rho(v(b)) &&=\sigma(\ac_v(a))/\sigma(\ac_v(b))\\
        &=w(\iota(a))-w(\iota(b)) &&=\ac_w(\iota(a))/\ac_w(\iota(b))\\
        &=w(\iota(a)/\iota(b)) &&=\ac_w(\iota(a)/\iota(b))\\
        &=w(\widetilde{\iota}(a/b)), &&=\ac_w(\widetilde{\iota}(a/b)).
    \end{align*}
    Since $A$ is closed under the parameterized $\l$-functions, then so is $F_0$ by Lemma \ref{lcltofrac}, i.e. $M|F_0$ is separable by Fact \ref{sylvy}. In order to prove that $\widetilde{\iota}$ commutes with the $\l$-functions, by Lemma \ref{cosepemb}, we just need to prove that $\widetilde{\iota}$ is a separable embedding, i.e. that $\widetilde{\iota} F_0$ is $\l$-closed in $N.$ Since $\widetilde{\iota} F_0=\Frac(\iota A),$ by Lemma \ref{lcltofrac}, it is enough to show that $\iota A$ is $\l$-closed in $N,$ but this follows from the fact that $\Sigma$ is an $\LL$-embedding. Indeed, if $b'=\iota b$ is a tuple in $\iota A$ which is $p$-independent in $N,$ if $a'=\iota a\in\iota A\cap N^{p}(b')$ and if $m\in\Mon(|b'|)$ is any monomial, then $$\l_m^{b'}(a')=\l_m^{\iota b}(\iota a)=\iota(\l_m^{b}(a))\in\iota A,$$
    as wanted. Note that $vF_0=vA$ and $F_0v=\Frac(Av).$ 

    \item: \label{step2} \emph{$\iota$ extends to $F_0^{h}$.} Indeed, if $i:(F_0,v)\to(F_0^{h},v)$ and $j:\iota F_0\to(\iota F_0)^{h}$ are the inclusions, then $i\circ\iota: (F,v)\to ((\iota F_0)^{h},w)$ is a henselization of $(F,v).$ Then, by the universal property of henselizations, there is a valued field embedding $\widetilde{\iota}:(F_0^{h},v)\to((\iota F_0)^{h},w)$ such that $j\circ\iota=\widetilde{\iota}\circ i,$ i.e. $\widetilde{\iota}$ is a valued field embedding extending $\iota.$
    
    Since the extension $F_0^{h}|F_0$ is immediate, it is in particular unramified and thus $\ac_v$ extends uniquely from $F_0$ to $F_0^{h}.$ Analogously, the extension $\widetilde{\iota}(F_0^{h})|\iota F_0$ is a sub-extension of the immediate extension $(\iota F_0)^{h}|\iota F_0$, so $\ac_w$ extends uniquely from $\iota F_0$ to $\widetilde{\iota}(F_0^{h}).$ It follows that $\sigma\circ \ac_v \circ(\widetilde{\iota}^{-1})$ is a well defined angular component on $\widetilde{\iota}(F_0^{h})$ extending $\ac_w$ on $\iota F_0,$ implying that $\ac_w=\sigma\circ \ac_v \circ(\widetilde{\iota}^{-1})$ on $\widetilde{\iota}(F_0^{h}),$ i.e. $$\widetilde{\iota}\circ\ac_w=\sigma\circ\ac_v.$$

    Finally, since $M|F_0$ is separable and $F_0^{h}|F_0$ is algebraic, then $M|F_0^{h}$ is also separable. Also, the extension $\widetilde{\iota}$ of $\iota$ from $F_0$ to $F_0^{h}$ is separable, because, for the same reasons, $\widetilde{\iota}(F_0^{h})|\iota F_0$ is algebraic as it is a sub-extension of the algebraic extension $(\iota F_0)^{h}|\iota F_0$ and $N|\iota F_0$ is separable. By Lemma \ref{cosepemb}, the ring embedding $\widetilde{\iota}$ commutes with the $\l$-functions of $N.$ 
    
    To sum up, the induced extension $\widetilde{\Sigma}=(\widetilde{\iota},\sigma,\rho)$ is a well founded $\LL$-embedding at this point. Therefore, replacing $F_0$ by $F_0^{h}$ and $\iota$ by $\widetilde{\iota},$ we may assume that $F_0$ is henselian.  
    
    \item: \label{Step 3} \emph{$\iota$ extends to some $F_1$ for which the extension $(F_1|F_0,v)$ is algebraic, unramified and such that $F_1v=(F_0v)^{rsc}$}. Indeed, let $S$ be the set of pairs $(F,\i_{*})$ where
    \begin{itemize}
        \item $F$ is a subfield of $M$ extending $F_0,$ 
        \item $F|F_0$ is algebraic,
        \item $\i_{*}:F\to N$ is a ring embedding for which $\Sigma_{*}=(\iota_{*},\sigma|_{Fv},\rho|_{vF}):\langle F\rangle\to\NN$ is an $\LL$-embedding that extends $\Sigma,$
        \item $Fv|F_0v$ is separable algebraic, and
        \item $vF=vF_0,$
    \end{itemize}
    ordered at each coordinate by inclusion. Then $S$ is closed under unions of chains,
    and $(F_0,\i)\in S.$ By Zorn's Lemma, there is a maximal element $(F_1,\i_1)\in S.$ We claim that $F_1v=(F_0v)^{rsc}=(F_1v)^{rsc}.$ Indeed, let $\a\in\kk(M)$ be separable-algebraic over $F_1v.$ Let $c_0,\.,c_n\in\O_v\cap F_1,$ with $c_n=1,$ be such that $\sum_{i=0}^{n}\res_v(c_i)X^{i}$ is the minimal polynomial of $\a$ over $F_1v.$ Since $(M,v)$ is henselian, there is some $a\in\O_v$ such that $\res_v(a)=\a$ and $\sum_{i=0}^{n}c_ia^{i}=0.$ It follows that $F_1(a)|F_1$ is algebraic, and since $F_1$ is henselian ---for it is an algebraic extension of $F_0$--, then $v$ extends uniquely to $F_1(a).$ Therefore, the inequalities 
    \begin{align*}
        n&=[F_1(a):F_1]\\
        &=p^{d}[F_1(a)v:F_1v](vF_1(a):vF_1)\\
        &\geq [F_1(a)v:F_1v]\\
        &\geq [(F_1v)(\a):F_1v]\\
        &=n
    \end{align*}
    imply that this extension is unramified, so $\ac_v$ extends uniquely to $F_1(a)$ too. If $b$ is a root of $\sum_{i=0}^{n}\i_1(c_i)X^{i}$ such that $\res_w(b)=\sigma(\a),$ we can define the embedding $\widetilde{\iota}$ of $F_1(a)$ in $N$ by $\widetilde{\iota}(a)=b$ and $\widetilde{\iota}=\i_1$ on $F_1.$ Note that the same argument implies that the restriction of $w$ to $(\i_1F_1)(b)=\widetilde{\iota}(F_1(a))$ is the unique extension of $w$ from $\i_1F_1$ to $(\i_1F_1)(b),$ and such unique extension is unramified. It follows that $$\rho\circ v \circ(\widetilde{\iota}^{-1})\text{ and }\sigma\circ \ac_v \circ(\widetilde{\iota}^{-1})$$ is a well defined couple of a valuation and an angular component on $(\i_1F_1)(b)$ extending $w$ and $\ac_w$ on $\i_1F_1,$ implying that $$w=\rho\circ v \circ(\widetilde{\iota}^{-1})\text{ and }\ac_w=\sigma\circ \ac_v \circ(\widetilde{\iota}^{-1})$$ on $(\i_1F_1)(b)$ by uniqueness. 

    Since $F_1(a)|F_1$ is algebraic, then $M|F_1(a)$ is separable. In order to see that $\widetilde{\iota}$ commutes with the $\l$-functions, we just need to prove that it is separable, as seen in Lemma \ref{cosepemb}. This follows from the equality $\widetilde{\iota}(F_1(a))=(\i_1F_1)(b),$ and the fact that $(\i_1F_1)(b)|\i_1F_1$ is an algebraic sub-extension of the separable extension $N|\i_1F_1.$ 

    To sum up, the induced extension $\widetilde{\Sigma}=(\widetilde{\iota},\sigma|_{F_1(a)v},\rho|_{vF_1(a)}):\langle F_1(a)\rangle\to \NN$ is a well founded $\LL$-embedding, so altogether $(F_1,\i_1)\leq(F_1(a),\widetilde{\iota})\in S,$ implying that $a\in F_1$ by maximality of $(F_1,\i_1)$ and that $\a=\res_v(a)\in F_1v,$ as wanted.

    Note that $F_1$ is henselian, being an algebraic extension of $F_0=F_0^{h}.$ Also, $F_1v$ is relatively separably closed in $\kk(M),$ and $M|F_1$ and $N|\widetilde{\iota} F_1$ are separable. Denote by $\Sigma_1=(\iota_1,\sigma|_{F_1v},\rho|_{vF_1}):\langle F_1\rangle\to\NN$ the corresponding $\LL$-embedding.

    \item: \label{Step 4} \emph{$\iota_1$ extends to the (ad-hoc) Artin-Schreier closure $F_2=F_1^{as}.$}
    First, let $S$ be the set of pairs $(F,\iota_{*})$ such that 
    \begin{itemize}
        \item $F$ is a subfield of $F_1(a\in F_1^{alg}:a^{p}-a\in F_1,\,v(a^{p}-a)<0)$ extending $F_1$,
        \item $\i_{*}:F\to N$ is a ring embedding for which $\Sigma_{*}=(\iota_{*},\sigma|_{Fv},\rho|_{vF}):\langle F\rangle\to\NN$ is an $\LL$-embedding that extends $\Sigma_1,$
    \end{itemize}
    ordered at each coordinate by inclusion. Then $S$ is closed under unions of chains, and $(F_1,\i_1)\in S.$ By Zorn's Lemma, there is a maximal element $(F_{1,0},\i_{1,0})\in S.$ We claim that $$F_{1,0}=F_1(a\in F_1^{alg}:a^{p}-a\in F_1,\,v(a^{p}-a)<0).$$
    To see this, let's start with some $c\in F_1$ such that $v(c)<0,$ and let $a\in F_1^{alg}$ be such that $a^{p}-a=c.$ Then $v(a)=v(c)/p.$ Since the Artin-Schreier polynomial $X^{p}-X-c$ is irreducible and $F_{1,0}$ is henselian, we have that $$p=[F_{1,0}(a):F_{1,0}]=p^{d}(vF_{1,0}(a):vF_{1,0})[F_{1,0}(a)v:F_{1,0}v]$$ for some $d\in\N.$ If $F_{1,0}(a)|F_{1,0}$ is unramified, as above, $\ac_v$ extends uniquely from $F_{1,0}$ to $F_{1,0}(a),$ and thus any extension $\widetilde{\iota}$ of $\iota_{1,0}$ taking $a$ to any Artin-Schreier root $b$ of $\iota_{1,0}(c)$ will define an $\LL_{\ac}$-embedding at this point.  
    
    If $F_{1,0}(a)|F_{1,0}$ is not unramified, then necessarily $(vF_{1,0}(a):vF_{1,0})=p$ and $v(a)+vF_{1,0}$ is a generator of the cyclic group $vF_{1,0}(a)/vF_{1,0}$ of order $p.$ Then the extension of $\ac_v$ from $F_{1,0}$ to $F_{1,0}(a)$ is completely determined by the value of $\ac_v(a).$ This means that we have to find some Artin-Schreier root of $b\in N$ of $\iota_{1,0}(c)$ such that $\ac_w(b)=\sigma\ac_v(a).$ We will show that this holds for \emph{any} Artin-Schreier root of $\iota_{1,0}(c).$
    Note that $a^{p}-a=c$ implies that $\ac_v(a^{p}-a)=\ac_v(c).$ Since $v(a^{p})<v(a)=v(c)/p<0,$ we get that $\ac_v(a^{p}-a)=\ac_v(a^{p})$ and thus $\ac_v(a)^{p}=\ac_v(c).$ Applying $\sigma$ and arguing similarly in $N$ mutatis mutandis, if $b$ is any Artin-Schreier root of $\iota_{1,0}(c),$ we obtain
    $$(\sigma\ac_v(a))^{p}=\sigma\ac_v(c)=\ac_w(\iota_{1,0}(c))=(\ac_w(b))^{p},$$
    meaning that $\sigma\ac_v(a)=\ac_w(b)$ as wanted. 

    Note that $F_{1,0}|F_1$ is an algebraic sub-extension of the separable extension $M|F_1,$ which implies that $M|F_{1,0}$ is separable. In order to prove that $\widetilde{\iota}$ commutes with the $\l$-functions, we have to prove that it is a separable embedding. This follows from the equality $\widetilde{\iota}(F_{1,0}(a))=(\iota_{1,0}F_{1,0})(b),$ from the fact that $M|F_{1,0}(a)$ is separable (as $F_{1,0}(a)|F_{1,0}$ is an algebraic sub-extension of the separable extension $M|F_{1,0}$) and from the fact that $(\iota_{1,0}F_{1,0})(b)|\iota_{1,0}F_{1,0}$ is an algebraic sub-extension of the separable extension $N|\iota_{1,0}F_{1,0}.$

    In any case, we conclude that $\widetilde{\i}:F_{1,0}(a)\to N$ induces a well-defined $\LL$-embedding $\widetilde{\Sigma}:\langle F_{1,0}(a)\rangle\to\NN,$ so $(F_{1,0},\iota_{1,0})\leq(F_{1,0}(a),\widetilde{\iota})\in S$ and $a\in F_{1,0}$ by maximality of $(F_{1,0},\iota_{1,0}),$ as wanted. Let $\Sigma_{1,0}=(\i_{1,0},\sigma|_{F_{1,0}v},\rho|_{vF_{1,0}}):\langle F_{1,0}\rangle\to\NN$ be the associated $\LL$-embedding. 
    Suppose that $F_{1,n},\,\iota_{1,n}$ and $\Sigma_{1,n}$ have been constructed for some $n\geq 0,$ and put 
    $$F_{1,n+1}=F_{1,n}(a\in F_{1,n}^{alg}:a^{p}-a\in F_{1,n}\text{ and }v(a^{p}-a)<0).$$ We can reproduce the same argument as above to conclude that $\iota_{1,n}$ can be extended to an embedding $\iota_{1,n+1}:F_{1,n+1}\to N$ that induces an $\LL$-embedding $\Sigma_{1,n+1}=(\iota_{1,n+1},\sigma|_{F_{1,n+1}v},\rho|_{vF_{1,n+1}}):\langle F_{1,{n+1}}\rangle\to\NN$ extending $\Sigma_{1,n}.$ Since $F_1^{as}=\bigcup_{n<\omega} F_{1,n},$ we conclude by induction that $\iota_1$ can be extended to some ring embedding $\iota_2$ of $F_2=F_1^{as}$ into $N$ inducing an $\LL$-embedding $\Sigma_2=(\iota_2,\sigma|_{F_2v},\rho|_{vF_2}):\langle F_2\rangle\to\NN$ that extends $\Sigma_1.$ 

    We get that $F_2v\subseteq (F_1v)^{rac}=((F_1v)^{rsc})^{1/p^{\infty}}=(F_1v)^{1/p^{\infty}}\subseteq F_2v,$ where the first inclusion follows because $F_1^{as}|F_1$ is an algebraic sub-extension of the extension $M|F_1$, the first equality follows because $\kk(M)$ is perfect, the second equality follows because $F_1v$ is relatively separably closed in $\kk(M)$, and the last inclusion follows from Lemma \ref{asclosure}. Thus, $F_2v$ is relatively algebraically closed in $\kk(M),$ as it coincides the relative algebraic closure of $F_1v$ in $\kk(M).$ 
    Also, we see that the value group of $F_2$ is $p$-divisible, and $M|F_2$ and $N|\iota_2F_2$ are separable extensions.       

    \item: \emph{$\iota_2$ extends to some $F_3$ such that $M|F_3$ is separable, $(F_3|F_2,v)$ is unramified and $F_3v=\kk(M)$.} Indeed, let $S$ be the set of pairs $(F,\iota_{*})$ such that 
     \begin{itemize}
        \item $F$ is a subfield of $M$ extending $F_2,$ 
        \item $M|F$ is separable,
        \item $\i_{*}:F\to N$ is a ring embedding for which $\Sigma_{*}=(\iota_{*},\sigma|_{Fv},\rho|_{vF}):\langle F\rangle\to\NN$ is an $\LL$-embedding that extends $\Sigma_2,$ and
        \item $vF$ is $p$-divisible, and 
        \item $Fv$ is relatively algebraically closed in $\kk(M),$
    \end{itemize}
    ordered at each coordinate by inclusion. Then $S$ is closed under unions of chains,\footnote{Say, if $(F_i)_{i\in I}$ is a chain of $\l$-closed subfields of $M$, then $\bigcup_{i\in I}F_i$ is a $\l$-closed subfield of $M$ too, so $M|\bigcup_{i\in I}F_i$ is separable.} and $(F_2,\i_2)\in S.$ By Zorn's Lemma, there is a maximal element $(F_3,\i_3)\in S.$ We claim that $F_3v=\kk(M).$ 
    Indeed, suppose $\a\in\kk(M).$ If $\a$ is algebraic over $F_3v,$ then $\a\in F_3v$ because $F_3v$ is relatively algebraically closed in $\kk(M).$ 
    
    If $\a$ is transcendental over $F_3v,$ use Lemma \ref{perfcoresat} to find some $a\in M$ with $v(a)=0,$ $F_3(a^{1/p^{\infty}})\subseteq M$ and such that $\res_v(a^{1/p^{n}})=\a^{1/p^{n}}$ for all $n\geq0.$ If $a$ is algebraic over $F_3,$ then $(F_3(a))v|F_3v$ and $(F_3v)(\a)|F_2v$ are algebraic too, because $(F_3v)(\a)\subseteq (F_3(a))v.$ Therefore $a$ and all its $p$-th roots $a^{1/p^{n}}$ are transcendental over $F_3.$      
    
    We conclude by Gauss's Extension Lemma \cite[Corollary 2.2.2]{ep} that the restriction of $v$ to $F_3(a)$ is the unique extension of $v$ from $F_3$ to $F_3(a),$ that this extension is unramified and that $(F_3(a))v=(F_3v)(\a)$. Use Lemma \ref{perfcoresat} to find some $b\in N$ such that $\res_w(b)=\sigma(\a)$ and $(\i_3F_3)(b^{1/p^{\infty}})\subseteq N.$ This can be done because $(N,v)$ is $\aleph_0$-saturated, $Nv$ is perfect and $\i_3$ is separable. 
    
    By the same reasoning, the restriction of $w$ to $(\i_3F_3)(b)$ is the unique extension of $w$ from $\i_3F_3$ to $(\i_3F_3)(b),$ and this extension is unramified. Define the extension $\i_{3,0}$ of $\i_3$ from $F_3$ to $F_3(a)$ by setting $\i_{3,0}(a)=b.$ As above, this extension induces an $\LL_{\ac}$-embedding, but $\i_{3,0}$ may not commute with the $\l$-functions as $M|F_3(a)$ and $N|(\i_3F_3)(b)$ may not be separable, i.e. $\i_{3,0}$ may not be separable.

    Put $F_{3,n}:=F_3(a,a^{1/p},\.,a^{1/p^{n}}),$ and suppose that its value group is $p$-divisible and $\i_{3,n}:F_{3,n}\to N$ induces an $\LL_{\ac}$-embedding. Note that these hypotheses hold for $F_{3,0}=F_3(a),$ because it is an unramified extension of $F_3,$ whose value group coincides with the $p$-divisible group $vF_2$.
    The algebraic extension $F_{3,n}(a^{1/p^{n+1}})|F_{3,n}$ is purely-inseparable of degree $p,$ so the restriction of $v$ to $F_{3,n}(a^{1/p^{n+1}})$ is the unique extension of $v$ from $F_{3,n}$ to $F_{3,n}(a^{1/p^{n+1}}).$ Therefore
    \begin{align*}
        p&=\left[F_{3,n}(a^{1/p^{n+1}}):F_{3,n}\right]=p^{d}\cdot\left(v\left(F_{3,n}(a^{1/p^{n+1}})\right):vF_{3,n}\right)\cdot\left[\left(F_{3,n}(a^{1/p^{n+1}})\right)v:F_{3,n}v\right]        
    \end{align*}
    for some $d\in\N.$ Note that $\left(v\left(F_{3,n}(a^{1/p^{n+1}})\right):vF_{3,n}\right)$ cannot be equal to $p,$ for $vF_{3,n}$ is $p$-divisible. It follows that $F_{3,n}(a^{1/p^{n+1}})|F_{3,n}$ is unramified, hence the restriction of $\ac_v$ to $F_{3,n}(a^{1/p^{n+1}})$ is the unique extension of $\ac_v$ from $F_{3,n}$ to $F_{3,n}(a^{1/p^{n+1}}).$
    
    Mutatis mutandis, we get that $$(\i_{3,n}(F_{3,n}))(b^{1/p^{n+1}})|\i_{3,n}(F_{3,n})$$ is an unramified extension with uniquely determined $w$ and $\ac_w.$ We can define the extension $\i_{3,n+1}$ of $\i_{3,n}$ on $F_{3,n}(a^{1/p^{n+1}})$ by setting $\i_{3,n+1}(a^{1/p^{n+1}})=b^{1/p^{n+1}},$ and this would induce a well founded $\LL_{\ac}$-embedding. 
    
    By induction, we get an $\LL_{\ac}$-embedding extending $\Sigma_2$ which is induced by some ring embedding $\i_{3,\infty}:F_3(a^{1/p^{\infty}})\to N,$ and the extension $F_3(a^{1/p^{\infty}})|F_3$ is unramified. By Lemma \ref{sepcore}, $M|F_3(a^{1/p^{\infty}})$ is separable. As above, if we wanted $\i_{3,\omega}$ to commute with the $\l$-functions, by Lemma \ref{cosepemb} we just need to prove that $\i_{3,\infty}$ is a separable embedding. Indeed, $\i_{3,\infty}\left(F_3(a^{1/p^{\infty}})\right)=(\i_3F_3)(b^{1/p^{\infty}}),$ and by Lemma \ref{sepcore} again, $N|(\i_3F_3)(b^{1/p^{\infty}})$ is separable. To sum up, $\i_{3,\infty}$ induces a well defined $\LL$-embedding $\Sigma_{3,\infty}:\langle F_3(a^{1/p^{\infty}})\rangle\to\NN$ that extends $\Sigma_2.$

    Now we are able to apply to $(F_3(a^{1/p^{\infty}}),\Sigma_{3,\infty})$ the arguments given in \ref{step2}, \ref{Step 3} and \ref{Step 4} in order to conclude that $\iota_{3,\infty}$ can be extended to some embedding $\widetilde{\iota}:F_3(a^{1/p^{\infty}})^{rac}\to N$ that induces an $\LL$-embedding $\widetilde{\Sigma}:\langle F_3(a^{1/p^{\infty}})^{rac}\rangle\to\NN$ that extends $\Sigma_{3,\infty},$ which in turn extends $\Sigma_2.$ With this, we can finally conclude that $(F_3,\i_3)\leq(F_3(a^{1/p^{\infty}})^{rac},\widetilde{\iota})\in S,$ and thus $a\in F_3$ by maximality of $(F_3,\i_3)$ in $S.$ This is absurd, as $a$ and $\a$ were transcendental over $F_3$ and $F_3v$ respectively. 

    Note that any further sub-extension $F$ of $M|F_3$ would satisfy that $Fv=F_3v=\kk(M).$ Also, $F_3|F_2$ may be transcendental, so $F_3$ may no longer be henselian. 

    \item: \label{Step6} \emph{$\i_3$ extends to $F_3^{h}$} as in \ref{step2}. Therefore, we may assume that $F_3$ is henselian, and nothing changes with respect to the residue fields and value groups.

    \item: \label{Step7} \emph{$\i_3$ extends to some $F_4$ such that $F_4|F_3$ is algebraic and $vF_4=(vF_3)^{rdh}$.} To this end, let $S$ be the set of pairs $(F,\iota_{*})$ such that 
    \begin{itemize}
        \item $F$ is a subfield of $M$ extending $F_3$ and $F|F_3$ is algebraic, 
        \item $\i_{*}:F\to N$ is a ring embedding for which $\Sigma_{*}=(\iota_{*},\sigma|_{Fv},\rho|_{vF}):\langle F\rangle\to\NN$ is an $\LL$-embedding that extends $\Sigma_3,$ and
        \item $vF$ is $p$-divisible,
    \end{itemize}
    ordered at each coordinate by inclusion. Then $S$ is closed under unions of chains, and $(F_3,\i_3)\in S.$ By Zorn's Lemma, there is a maximal element $(F_4,\i_4)\in S.$ Note that $(F_4,v)$ is henselian for it is an algebraic extension of the henselian valued field $(F_3,v).$ We claim that $vF_4=(vF_4)^{rdh}=(vF_3)^{rdh}.$ Since $vF_4\subseteq (vF_3)^{rdh}$ given that $F4|F_3$ is algebraic, we just need to prove the inclusion $(vF_4)^{rdh}\subseteq vF_4.$ 
    To this end, let $\g\in(vF_4)^{rdh}$ and let $n$ be the minimum natural number for which $n\g\in vF_4.$ If $p>1$ and $n=p^{d}m$ for some power $d>0$ and some $m$ coprime with $p,$ then $m\g\in vF_4$ because $vF_4$ is $p$-divisible, contradicting minimality of $n.$ Hence $n$ is coprime with $p$ regardless of $p$ being $1$ or greater than $1.$ By Lemma \ref{nwhit}, there is some $a\in M$ such that $v(a)=\g,$ $c:=a^{n}\in F_4$ and $\ac_v(a)=1.$ By minimality of $n,$ the polynomial $X^{n}-c\in F_4[X]$ is the minimal polynomial of $a$ over $F_4.$\footnote{If $\sum_{i=0}^{d}c_iX^{i}$ is the minimal polynomial of $a$ over $F_4$ (with $c_d=1$) and $d<n,$ then $v(c_i)+i\g\neq v(c_j)+j\g$ for all $0\leq i<j\leq d$ by minimality of $n$. But then $d\g=v(c_i)+i\g$ for some $i\in\{0,\.,d-1\},$ contradicting minimality of $n.$} Therefore, given that $F_4$ is henselian, $n=[F_4(a):F_4]=p^{d}(v(F_4(a)):vF_4)$ for some $d\geq0,$ and because $n$ and $p$ are coprime we get that $n=(v(F_4(a)):vF_4)$. Minimality of $n$ also implies that $v(F_4(a))/vF_4$ is cyclic of order $n,$ generated by $\g+vF_4=v(a)+vF_4.$\footnote{Indeed, the cosets $\{vF_4,\g+vF_4,\.,(n-1)\g+vF_4\}$ are pairwise different by minimality of $n,$ so $\g+vF_4$ is an element of order $n$ in a group $v(F_4(a))/vF_4$ of order $n,$ i.e. $v(F_4(a))/vF_4$ is cyclic and generated by $\g+vF_4.$} 
    
    It follows that the extension of $\ac_v$ from $F_4$ to $F_4(a)$ is completely determined by the value of $\ac_v(a).$ Hence, our goal is to find an element $b\in N$ such that $w(b)=\rho\g,$ $\ac_w(b)=\sigma(\ac_v(a))=\sigma(1)=1$ and $b^{n}=\i_4(c),$ for in this case, the extension $\widetilde{\iota}$ of $\i_4$ defined on $F_4(a)$ and given by $\widetilde{\iota}(a)=b$ would induce a well defined $\LL_{\ac}$-embedding. 

    To this end, find some $d\in N$ such that $w(d)=\rho\g$ and $\ac_w(d)=1,$ say, by Lemma \ref{nwhit}. Since $v(c)=n\g$ and $\ac_v(c)=(\ac_v(a))^{n}=1,$ applying $\rho$ we get that $n\rho\g=\rho(v(c))=w(\i_4(c)),$ so that $w(b^{n}/\i_4(c))=0$ and $\res_w(d^{n}/\i_4(c))=\ac_w(d^{n}/\i_4(c))=(\ac_w(d))^{n}/\ac_w(\i_4(c))=1/\sigma(\ac_v(c))=1.$ By henselianity of $(N,w),$ we get that there is some $f\in N$ such that $f^{n}=d^{n}/\i_4(c)$ and $w(f-1)>0.$ Thus $\ac_w(f)=\res_w(f)=1,$ and $(d/f)^{n}=\i_4(c)\in\i_4F_4,$ $w(d/f)=w(d)=\rho\g$ and $\ac_w(d/f)=1,$ as wanted. 

    Now, if we wanted $\widetilde{\iota}$ to commute with the $\l$-functions, we need to check that $M|F_4(a)$ is separable and that $\widetilde{\iota}$ is separable. Since $M|F_3$ is separable and $F_4(a)|F_3$ is an algebraic sub-extension thereof, we get that $M|F_4(a)$ is indeed separable. Analogously, since $\widetilde{\iota}(F_4(a))=(\i_4F_4)(b)$ is an algebraic extension of $\i_3F_3$ and $N|\i_3F_3$ is separable, we also get that $N|(\i_4F_4)(b)$ is separable, i.e. that $\widetilde{\iota}$ is separable. It follows that there is some well defined $\LL$-embedding $\widetilde{\Sigma}:\langle F_4(a)\rangle\to\NN$ induced by $\widetilde{\i}$ that extends $\Sigma_3.$ 
    
    However, it is not guaranteed that $(F_4(a),\widetilde{\iota})\in S$ because $v(F_4(a))$ may not be $p$-divisible, so we may further extend $(F_4(a),\widetilde{\iota})$ to its Artin-Shreier closure. This can be done by repeating the argument given in \ref{Step 4} for $(F_4(a),\widetilde{\iota})$ instead of $(F_1,\i_1).$ Namely, we can find an extension $\i':F_4(a)^{as}\to N$ that induces an $\LL$-embedding $\Sigma':\langle F_4(a)^{as}\rangle\to\NN$ extending $\widetilde{\Sigma}.$  
    Finally, it follows that $(F_4,\i_4)\leq(F_4(a)^{as},\i')\in S,$ so $a\in F_4$ by maximality of $(F_4,\iota_4)$ and $\g=v(a)\in vF_4,$ as wanted. 

    Note that, at this point, the value group $vF_4$ is relatively divisible in $\GG(M)$, for it coincides with the relative divisible hull of $vF_3.$

    \item: \emph{$\i_4$ extends to a some sub-extension $F_5|F_4$ for which $M|F_5$ is separable and $vF_5=\GG(M).$} To see this, let $S$ be the set of pairs $(F,\i_{*})$ such that 
    \begin{itemize}
        \item $F$ is a subfield of $M$ extending $F_4,$ 
        \item $M|F$ is separable, 
        \item $\i_{*}:F\to N$ is a ring embedding for which $\Sigma_{*}=(\iota_{*},\sigma|_{Fv},\rho|_{vF}):\langle F\rangle\to\NN$ is an $\LL$-embedding that extends $\Sigma_4,$
        \item $vF$ is relatively divisible in $\GG(M),$
    \end{itemize}
    ordered at each coordinate by inclusion. Then $S$ is closed under unions of chains, and $(F_4,\i_4)\in S.$ By Zorn's Lemma, there is a maximal element $(F_5,\i_5)\in S.$ We claim that $vF_5=\GG(M).$ 
    
    Indeed, let $\g\in\GG(M).$ Use Lemma \ref{perfcoresat} to find some $a\in M$ such that $F_5(a^{1/p^{\infty}})\subseteq M$ and $v(a^{1/p^{n}})=\g/p^{n}$ for all $n\geq0.$ 
    
    If $a$ is algebraic over $F_5,$ then $v(F_5(a))\subseteq (vF_5)^{rdh}=vF_5$ and $\g\in vF_5.$ Also, if there is some $n\geq 1$ such that $n\g\in vF_5,$ then $\g\in vF_5$ because $vF_5$ is relatively divisible in $\GG(M).$   
    If $a$ is transcendental over $F_5$ and the only $n\in\N$ satisfying $n\g\in vF_5$ is $n=0,$ then the extension of $v$ from $F_5$ to $F_5(a)$ is uniquely determined and coincides with the Gauss Extension of $v|_{F_5},$ by Gauss's Extension Lemma \cite[Corollary 2.2.3]{ep}. This extension is not unramified, because in this case $v(F_5(a))=vF_5\oplus\Z\g.$ This implies that $vF_5(a)/vF_5$ is generated by $\g=v(a),$ so the extension of $\ac_v$ from $F_5$ to $F_5(a)$ is uniquely determined by the value of $\ac_v(a).$ This is, if we find some $b\in N$ such that $w(b)=\rho\g$ and $\ac_w(b)=\sigma(\ac_v(a)),$ then, as above, the extension $\widetilde{\iota}$ of $\i_5$ given by $\widetilde{\iota}(a)=b$ induces a well founded $\LL_{\ac}$-embedding defined in $F_5(a).$ We will even find some $b\in N$ such that $(\i_5F_5)(b^{1/p^{n}})\subseteq N$ and such that $\ac_w(b^{1/p^{n}})=\sigma(\ac_v(a))^{1/p^{n}}$ and $w(b^{1/p^{n}})=\rho\g/p^{n}$ for all $n<\omega.$

    Since $(N,w,\ac_w)$ is $\aleph_0$-saturated, we can obtain such a $b$ by realizing the partial type $\pi(x)$ over \{$\rho\g,\sigma(\ac_v(a))\}$ given by $$\left\{\E y\left(y^{p^{n}}=x\wedge p^{n}\underline{v}(y)=\g\wedge\underline{\ac}(y)^{p^{n}}=\sigma(\ac_v(a))\right):n<\omega\right\},$$ once we prove that $\pi(x)$ is finitely satisfiable.  
    
    To this end, let $n<\omega$. Note that $\ac_v(a)^{1/p^{n}}\in Mv=F_5v,$ so there is some $c\in F_5$ such that $v(c)=0$ and $\ac_v(a)^{1/p^{n}}=\res_v(c).$ Applying $\sigma,$ we get that $\sigma(\ac_v(a))^{1/p^{n}}=\sigma(\res_v(c))=\res_w(\iota(c))=\ac_w(\iota(c)),$ i.e. $$\dfrac{\sigma(\ac_v(a))^{1/p^{n}}}{\ac_w(\iota(c))}=1.$$
    Likewise, let $b\in N$ be such that $w(b)=\rho\g/p^{n}.$ Then $\ac_w(b)=\res_w(d)=\ac_w(d)$ for some $d\in N$ with $w(d)=0,$ so $w(b/d)=w(b)=\rho\g/p^{n}$ and $$\ac_w\left(\dfrac{b}{d}\right)=1=\dfrac{\sigma(\ac_v(a))^{1/p^{n}}}{\ac_w(\iota(c))}.$$ 
    Therefore $$\ac_w\left(\iota(c)\cdot\dfrac{b}{d}\right)=\sigma(\ac_v(a))^{1/p^{n}}\text{ and }w\left(\iota(c)\cdot\dfrac{b}{d}\right)=w(b)=\rho\g/p^{n}.$$ 
    If $e=\iota(c)\cdot\dfrac{b}{d},$ then for all $i\in\{0,\.,n\}$ we would have that 
    $$\ac_w\left(e^{p^{n-i}}\right)=\sigma(\ac_v(a))^{1/p^{i}}\text{ and }w\left(e^{p^{n-i}}\right)=\rho\g/p^{i},$$ this is, $e^{p^{n}}$ is a realization of the finite subset $$\left\{\E y\left(y^{p^{i}}=x\wedge p^{i}\underline{v}(y)=\g\wedge\underline{\ac}(y)^{p^{i}}=\sigma(\ac_v(a))\right):i\in\{0,\.,n\}\right\}\subseteq\pi(x),$$ as wanted.

    Put $F_{5,n}:=F_5(a,a^{1/p},\.,a^{1/p^{n}}),$ and suppose that its value group is $p$-divisible and that $\i_{5,n}:F_{5,n}\to N$ induces an $\LL_{\ac}$-embedding $\Sigma_{5,n}$ that extends $\Sigma_5$. Note that these hypotheses hold for $F_{5,0}=F_5(a)$ and $\i_{5,0}=\widetilde{\i},$ given that $vF_5$ is $p$-divisible as it is relatively divisible in the $p$-divisible group $\GG(M).$ 
    The algebraic extension $F_{5,n}(a^{1/p^{n+1}})|F_{5,n}$ is purely-inseparable of degree $p,$ so the restriction of $v$ to $F_{5,n}(a^{1/p^{n+1}})$ is the unique extension of $v$ from $F_{5,n}$ to $F_{5,n}(a^{1/p^{n+1}}).$ Therefore, since $F_5v=Mv,$ $$p=\left[F_{5,n}(a^{1/p^{n+1}}):F_{5,n}\right]=p^{d}\cdot\left(v\left(F_{5,n}(a^{1/p^{n+1}})\right):vF_{5,n}\right)$$ for some $d\geq0.$
    Since the value group of $F_{5,n}$ is $p$-divisible, we get that the extension $F_{5,n}(a^{1/p^{n+1}})|F_{5,n}$ must be unramified. Mutatis mutandis, we get that $$\iota(F_{5,n})(b^{1/p^{n+1}})|\iota(F_{5,n})$$ is an immediate extension with uniquely determined $w$ and $\ac_w.$ We can define the extension $\i_{5,n+1}$ of $\i_{5,n}$ on $F_{5,n}(a^{1/p^{n+1}})$ by setting $\widetilde{\iota}(a^{1/p^{n+1}})=b^{1/p^{n+1}},$ and this would induce a well founded $\LL_{\ac}$-embedding extending $\Sigma_{5,n}$. By induction, we get an $\LL_{\ac}$-embedding $\Sigma_{5,\infty}$ induced by $\i_{5,\infty}:F_5(a^{1/p^{\infty}})\to N,$ that extends $\Sigma_5,$ and the extension $F_5(a^{1/p^{\infty}})|F_5$ is unramified. By Lemma \ref{sepcore}, $M|F_5(a^{1/p^{\infty}})$ is separable. As above, if we wanted $\i_{5,\infty}$ to commute with the $\l$-functions, by Lemma \ref{cosepemb} we just need to prove that $\i_{5,\infty}$ is a separable embedding. Indeed, $\i_{5,\infty}\left(F_5(a^{1/p^{\infty}})\right)=(\i_5F_5)(b^{1/p^{\infty}}),$ and by Lemma \ref{sepcore} again, $N|(\i_5F_5)(b^{1/p^{\infty}})$ is separable. Thus, $\iota_{5,\infty}$ induces a well-defined $\LL$-embedding $\Sigma_{5,\infty}:\langle F_{5,\infty}\rangle\to\NN$ that extends $\Sigma_5.$  

    Since $v(F_5(a^{1/p^{\infty}}))$ may not be relatively divisible in $\GG(M),$ we can repeat the arguments given in \ref{Step6} and \ref{Step7}, replacing $(F_3,\i_3)$ by $(F_5(a^{1/p^{\infty}}),\i_{5,\infty}),$ in order to get some algebraic extension $F_5'|F_5(a^{1/p^{\infty}})$ contained in $M$ together with an extension $\i':F_5'\to N$ of $\i_{5,\infty}$ that induces an $\LL$-embedding $\Sigma':\langle F_5'\rangle\to\NN$ which extends $\Sigma_{5,\infty}$, and such that $vF_5'=v(F_5(a^{1/p^{\infty}}))^{rdh}.$ With this, we get finally that $(F_5,\iota_5)\leq(F_5',\i_5')\in S,$ so $a\in F_5$ by maximality of $(F_5,\i_5).$ This is absurd because $a$ was supposed to be transcendental over $F_5,$ so this case is empty. 

    Note that $F_5$ may not be an algebraic extension of $F_4,$ so $F_5$ may not be henselian. Also, given that $vF_5=\GG(M)$ and $F_5v=\kk(M),$ the extension $M|F_5$ is immediate, and $(F_5,v)$ is Kaplansky.

    \item: \emph{$\iota_5$ extends to $F_5^{h}$} as in \ref{step2}. Therefore, we may assume that $F_5$ is henselian, and nothing changes with respect to the residue fields and value groups

    \item: \emph{$\iota_5$ extends to $F_6=F_5^{rac}.$} First note that $F_5^{rac}=F_5^{rsc}$ because $F_5^{rac}|F_5$ is a sub-extension of the separable extension $M|F_5.$ Second, note that $(F_5^{rac},v)$ is separably algebraically maximal by Lemma \ref{pop}, for $M|F_5^{rac}$ is regular. Let $F_5^{sep}\subseteq F_5^{alg}$ be some fix separable and algebraic closures of $F_5$ respectively, for which $F_5^{rsc}=F_5^{sep}\cap M.$ Let $v'$ be an extension of $v|_{F_5}$ to $F_5^{alg},$ and let $(K,v')\subseteq (F_5^{alg},v')$ be a maximal algebraic immediate extension of $(F_5^{rsc},v).$ It follows that $F_5^{rsc}=F_5^{sep}\cap K,$ since $F_5^{sep}\cap K$ is an immediate separably algebraic extension of $(F_5^{rsc},v)=(F_5^{rac},v)$ and $(F_5^{rac},v)$ is separably algebraically maximal.

    Analogously, if $((\i_5F_5)^{alg},w')$ is an extension of $(\i_5F_5,w)$ and $(K',w')\subseteq((\i_5F_5)^{alg},w')$ is a maximal immediate algebraic extension of $((\i_5F_5)^{rsc},w),$ we conclude that $(\i_5F_5)^{rsc}=(\i_5F_5)^{sep}\cap K'.$ Since $F_5$ and $\i_5F_5$ are henselian and the extensions $(F_5^{rsc}|F_5,v)$ and $((\i_5F_5)^{rsc}|\i_5F_5,w)$ are unramified, we conclude that the corresponding extensions of $v,$ $\ac_v,$ $w$ and $\ac_w$ are unique. Of most importance, by Kaplansky's Uniqueness Theorem \ref{kap}, $\i_5:F_5\to\i_5F_5$ can be extended to a ring isomorphism $\i_6:F_5^{sep}\cap K\to(\i_5F_5)^{sep}\cap K'$ that induces an $\LL_{\ac}$-embedding $\Sigma_{6}:\langle F_5^{rac}\rangle\to\NN.$ We also get that $\i_6$ is a separable embedding, for $M|F_5^{rac}$ and $N|(\i_5F_5)^{rac}$ are separable and $\i_6(F_5^{rac})=(\i_5F_5)^{rac},$ hence $\Sigma_6$ is even an $\LL$-embedding, as wanted. 
    
    \item: \emph{Let $\ol{a}$ be a $p$-basis of $M$ over $F_6.$ Then $\iota_6$ extends to $F_7=F_6(\ol{a})^{rac}.$} 
    To see this, let $S$ be the set of pairs $(\ol{c},\i_{*})$ such that
    \begin{itemize}
        \item $\ol{c}$ is a sub-tuple of $\ol{a},$ and

        \item $\i_{*}:F_6(\ol{c})^{rac}\to N$ is a ring embedding for which $\Sigma_{*}=(\iota_{*},\sigma|_{Fv},\rho|_{vF}):\langle F\rangle\to\NN$ is an $\LL$-embedding that extends $\Sigma_6,$ 
    \end{itemize}

    ordered at each coordinate by inclusion. Then $S$ is closed under unions of chains, and $(\emptyset,\i_6)\in S.$ By Zorn's Lemma, there is a maximal element $(\ol{c},\i_7)\in S.$ Name $F_7:=F_6(\ol{c})^{rac},$ and note that $(F_7,v)$ is separably algebraically maximal Kaplansky by Lemma \ref{pop}, because $M|F_7$ is regular and immediate. We now claim that $F_7=F_6(\ol{a})^{rac},$ i.e. that $\ol{c}=\ol{a}.$ If this is not the case, then there is some $a\in\ol{a}\setminus\ol{c},$ making $F_7(a)|F_7$ is transcendental by Lemma \ref{pindtran}. 
    Since $(F_7(a)|F_7,v)$ is immediate, there is some pseudo-Cauchy sequence $\{a_\nu\}_\nu\subseteq F_7$ without pseudo-limit in $F_7$ such that $\{a_\nu\}_\nu$ pseudo-converges to $a.$ 
    
    \emph{Case 1: $\{a_\nu\}_\nu$ is of transcendental type over $F_7.$} By statement 1 of Fact \ref{tralgtype}, the extension of $v$ from $F_7$ to $F_7(a)$ is uniquely determined, as well as the corresponding extension of $\ac_v$, given that $F_7(a)|F_7$ is in particular unramified.
    The sequence $\{\i_7(a_\nu)\}_\nu\subseteq \i_7F_7$ is necessarily of transcendental type as well. Since $N$ is $|M|^{+}$-saturated, there is some $b_a\in N$ which is a pseudo-limit of $\{\i_7(a_\nu)\}_\nu.$ Then the extension $\widetilde{\iota}$ of $\i_7$ given by $\widetilde{\iota}(a)=b_a$ is, as above, an $\LL_{\ac}$-embedding, again by the uniqueness showed in statement 1 of Fact \ref{tralgtype}. Note that we chose $\widetilde{\iota}(a)=b_a$ to be a pseudo-limit of the sequence $\{\i_7(a_\nu)\}_\nu,$ but \emph{any} pseudo-limit of such sequence would be a suitable choice for $\widetilde{\iota}(a)$, i.e. any element $b'$ satisfying $w(b'-b_a)>w(b_a-\i_7(a_\nu))$ for all $\nu$ or, equivalently, satisfying that $$w(b'-b_a)\geq\Xi_a$$ where $\Xi_a>\sup_{\nu}w(b_a-\i_7(a_\nu)).$ By saturation of $(N,w),$ there is an element $\Xi_a\in wN$ satisfying this property.     
    If we want $\widetilde{\iota}$ to commute with the $\l$-functions, by Lemma \ref{sepcorr} we also have to choose $\widetilde{\iota}(a)$ outside the $p$-closure of $\i_7\ol{c}$ in $N$ over $\i_7F_6,$ i.e. in $N\setminus N^{p}(\i_7F_6)(\i_7\ol{c}).$ By Lemma \ref{vtop}, we just need to argue that the field extension $N|N^{p}(\i_7F_6)(\i_7\ol{c})$ is proper.     

    For any field $K$ of characteristic exponent $p,$ we will write $\imdeg(K):=\imdeg(K|\F_p)$ to denote the (absolute) imperfection degree of $K,$ where $\F_p$ denotes the prime field of $K.$ Since $\ol{a}$ is a $p$-basis of $M$ over $F_7,$ by Lemma \ref{comppind}, we get that $$\begin{cases}
        \imdeg(M)=|\ol{a}|+\imdeg(F_7),\\
        \imdeg(F_7)=|\ol{c}|+\imdeg(F_6), \text{ and}\\
        \imdeg(F_7(a))=|\ol{c}|+1+\imdeg(F_6).
    \end{cases}$$ 
    By hypothesis, $\MM$ and $\NN$ share the same Ershov degree $e.$ This means that either $e=\imdeg(M)=\imdeg(N)$ is finite, or both $\imdeg(M)$ and $\imdeg(N)$ are infinite, in which case $\imdeg(N)>\imdeg(M)$ by saturation of $N$.
    If $N=N^{p}(\i_7F_6)(\i_7\ol{c}),$ given that $\i_7\ol{c}$ is $p$-independent in $N$, we get that 
    \begin{align*}
        \imdeg(N)&=|\i_7\ol{c}|+\imdeg(\i_7F_6)&\imdeg(N)&=|\i_7\ol{c}|+\imdeg(\i_7F_6)\\
        &=|\ol{c}|+\imdeg(F_6)&&=|\ol{c}|+\imdeg(F_6)\\
        &=\imdeg(F_7)&&=\imdeg(F_7)\\
        &<\imdeg(F_7(a))&&=\imdeg(F_7(a))\\
        &\leq \imdeg(M),&&\leq \imdeg(M),
    \end{align*}
    if $|\ol{c}|$ is finite or infinite respectively. Therefore we reach a contradiction if $\imdeg(M)=\imdeg(N)$ is finite, by the strict inequality of the left-hand side column, and we get a contradiction if $\imdeg(M)<\imdeg(N)$ and $|\ol{c}|$ are all infinite, by the inequality of the right-hand side column. 
    
    Now, up to this point, we were able to choose the right $\widetilde{\iota}(a)$ for which $\widetilde{\i}$ induces a well-founded $\LL$-embedding $\widetilde{\Sigma}:\langle F_7(a)\rangle\to\NN$ that extends $\Sigma_7.$ Since the fields $(F_7(a),v)$ and $(\i_7(F_7)(\widetilde{\i}a),w)$ are Kaplansky and the extensions $(F_7(a)^{rac}|F_7(a),v)$ and $\Big((\i_7F_7)(\widetilde{\iota}(a))^{rac}|(\i_7F_7)(\widetilde{\iota}(a)),w\Big)$ are immediate, by Kaplansky's Theorem of uniqueness of maximal algebraic immediate extensions of Kaplansky fields \ref{kap}, we can extend $\widetilde{\iota}$ to an $\LL_{\ac}$-embedding $\i'$ defined over $F_7(a)^{rac},$ and since the extensions are algebraic and $\widetilde{\iota}$ is separable, we get that $\i'$ is separable too. We can finally conclude that $(\ol{c},\i_7)\leq(a^{\frown}\ol{c},\i')\in S,$ so $a\in\ol{c}$ by maximality of $(\ol{c},\i_7).$ But this is absurd, as $a$ was supposed to be an element in $\ol{a}\setminus\ol{c}.$  

    \emph{Case 2: $\{a_\nu\}_\nu$ is of algebraic type over $F_7.$} By Fact \ref{sampc}, we get that $a\in F_7^{c}.$ By Remark \ref{trt}, we may replace $\{a_\nu\}_\nu$ by some Cauchy sequence of \emph{transcendental} type that converges to $a,$ and we can reach a contradiction by arguing as in Case 1, as wanted.
    Note that $M$ and $F_7=F_6(\ol{a})^{rac}$ have the same imperfection degree. 

    \item: \emph{$\iota_7$ extends to $M.$} Indeed, let $S$ be the set of pairs $(F,\i_{*})$ such that 
    \begin{itemize}
        \item $F$ is a subfield of $M$ extending $F_7,$ 
        \item $\i_{*}:F\to N$ is a ring embedding for which $\Sigma_{*}=(\iota_{*},\sigma|_{Fv},\rho|_{vF}):\langle F\rangle\to\NN$ is an $\LL_{\ac}$-embedding that extends $\Sigma_7,$ and
        \item $F$ is relatively algebraically closed in $M,$
    \end{itemize}
    ordered at each coordinate by inclusion.\footnote{Note that, in this step, we only need $\Sigma^{*}$ to be an $\LL_{\ac}$-embedding, instead of a full $\LL$-embedding.} Then $S$ is closed under unions of chains, and $(F_7,\i_7)\in S.$ By Zorn's Lemma, there is a maximal element $(F_8,\i_8)\in S.$ 
    
    Note that $\langle F_8\rangle\models\texttt{SAMK}_e^{\l,\ac}$ by Lemma \ref{popnotsep}, because $F_8$ is relatively algebraically closed in $M$, $M|F_8$ is immediate and the imperfection degree of $F_8$ is the same as the one of $M.$ We claim that $F_8=M.$ Since $F_8$ is relatively algebraically closed in $M,$ it is enough to prove that $M|F_8$ is algebraic.  
    If this is not the case, then there is some $a\in M$ such that $F_8(a)|F_8$ is transcendental. 
    Since $(F_8(a)|F_8,v)$ is immediate, there is some pseudo-Cauchy sequence $\{a_\nu\}_\nu\subseteq F_8$ without pseudo-limit in $F_8$ such that $\{a_\nu\}_\nu$ pseudo-converges to $a.$ 
    
    \emph{Case 1: $\{a_\nu\}_\nu$ is of transcendental type over $F_8.$} By statement 1 of Fact \ref{tralgtype}, the extension of $v$ from $F_8$ to $F_8(a)$ is uniquely determined, as well as the corresponding extension of $\ac_v$, given that $F_8(a)|F_8$ is in particular unramified.
    The sequence $\{\i_8(a_\nu)\}_\nu\subseteq \i_8F_8$ is necessarily of transcendental type as well. Since $N$ is $|M|^{+}$-saturated, there is some $b\in N$ which is a pseudo-limit of $\{\i_8(a_\nu)\}_\nu.$ Then the extension $\widetilde{\iota}$ of $\i_8$ given by $\widetilde{\iota}(a)=b$ is an $\LL_{\ac}$-embedding, by the uniqueness showed in statement 1 of Fact \ref{tralgtype}. 
    
    Now, up to this point, we were able to choose the right $\widetilde{\iota}(a)$ for which $\widetilde{\i}$ induces a well-founded $\LL_{\ac}$-embedding $\widetilde{\Sigma}:\langle F_8(a)\rangle\to\NN$ that extends $\Sigma_8.$ Since the fields $(F_8(a),v)$ and $((\i_8F_8)(b),w)$ are Kaplansky and the extensions $(F_8(a)^{rac}|F_8(a),v)$ and $\Big((\i_8F_8)(b)^{rac}|(\i_8F_8)(b),w\Big)$ are immediate, by Kaplansky's Uniqueness Theorem \ref{kap}, we can extend $\widetilde{\iota}$ to an $\LL_{\ac}$-embedding $\i'$ defined over $F_8(a)^{rac},$ and since the extensions are algebraic and $\widetilde{\iota}$ is separable, we get that $\i'$ is separable too. We can finally conclude that $(F_8,\i_8)\leq(F_8(a)^{rac},\i')\in S,$ so $a\in F_8$ by maximality of $(F_8,\i_8).$ But this is absurd, as $a$ was supposed to be transcendental over $F_8.$ 

    \emph{Case 2: $\{a_\nu\}_\nu$ is of algebraic type over $F_8.$} By Fact \ref{sampc}, we get that $a\in F_8^{c}.$ By Remark \ref{trt}, we may replace $\{a_\nu\}_\nu$ by some Cauchy sequence of \emph{transcendental} type that converges to $a,$ and reach a contradiction by arguing as in Case 1.

    We conclude that $\i_8:M\to N$ is a ring embedding extending $\i_7:F_7\to N.$ Since $M|F_7$ is separable and $M$ and $F_7$ have the same imperfection degree, we conclude that $\i_8$ is separable, by Lemma \ref{fincosep}. Since it also induces an $\LL_{ac}$-embedding $\Sigma_8:\MM\to\NN,$ we conclude that $\Sigma_8$ is even an $\LL$-embedding, as wanted.
\qedhere
\end{enumerate}
\end{proof}

\section{Relative Quantifier Elimination}

Recall that $\LL_\kk$ and $\LL_\GG$ are allowed to have more symbols than $\LL_{ring}$ and $\LL_{og}\cup\{\infty\}$ respectively. If $\LL=\LL(\LL_\kk,\,\LL_\GG),$ then for a finite tuple $y=(y_1,\.,y_n)$ of variables of sort $\kk$ and an $\LL_\kk\,$-formula $\th(y_1,\.,y_n),$ let $R_{\th}(y)$ be a new $n$-ary relation symbol. Respectively, for a finite tuple $z=(z_1,\.,z_n)$ of variables of sort $\GG$ and an $\LL_\GG\,$-formula $\psi(z_1,\.,z_n),$ let $R_{\psi}(z)$ be a new $n$-ary relation symbol.
Let $\LL^{+}$ be the expansion of $\LL$ by the symbols $R_{\th}$ and $R_{\psi}.$ Finally, let $T$ be the $\LL^{+}$-theory given by the union of $\texttt{SAMK}_{e}^{\l,\ac}$ and the axioms $$\begin{cases}
    \A y\Big(R_{\th}(y)\sii\th(y)\Big),\\
    \A z\Big(R_{\psi}(z)\sii\psi(z)\Big)
\end{cases}$$ 
for all $\th$ and $\psi$ as above. In other words, $T$ consist of the morleyization of $\texttt{SAMK}_e^{\l,\ac}$ with respect to $\LL_\kk$ and $\LL_\GG.$ 
\begin{propo}
\label{rqe}
    $T$ has quantifier elimination.
\end{propo}

\begin{remark}
\label{respleq}
    Since $\LL_\kk$ and $\LL_\GG$ are allowed to have more symbols than $\LL_{ring}$ and $\LL_{og}\cup\{\infty\}$ respectively, then Proposition \ref{rqe} says that the $\LL(\LL_{ring},\LL_{og}\cup\{\infty\})$-theory $\texttt{SAMK}_{e}^{\l,\ac}$ \emph{resplendently eliminates quantifiers from the field sort $\KK$}, cf. \cite[Definition A.6]{r}.
\end{remark}

For the proof of Proposition \ref{rqe}, we will use the following well known test for quantifier elimination, which is suitable for an application of the Embedding Lemma \ref{emblem}.
\begin{fact}
\label{qetest}
    The following statements are equivalent.
    \begin{enumerate}[label*={\arabic*.}]
        \item $T$ has quantifier elimination.

        \item For all models $\MM,\NN$ of $T,$ if $\MM$ is countable, $\NN$ is $\aleph_1$-saturated, $\AA$ is a substructure of $\MM$ and $\Sigma:\AA\to\NN$ is an $\LL^{+}$-embedding, then there is some $\LL^{+}$-embedding $\widetilde{\Sigma}:\MM\to\NN$ whose restriction to $\AA$ coincides with $\Sigma.$

        \item For all models $\MM,\NN$ of $T,$ if $\MM$ is $\aleph_0$-saturated, $\NN$ is $|M|^{+}$-saturated, $\AA$ is a substructure of $\MM$ and $\Sigma:\AA\to\NN$ is an $\LL^{+}$-embedding, then there is some $\LL^{+}$-embedding $\widetilde{\Sigma}:\MM\to\NN$ whose restriction to $\AA$ coincides with $\Sigma.$
    \end{enumerate}
\end{fact} 

\begin{proof}[Proof of Proposition \ref{rqe}]
Let $\MM,\NN,\AA$ and $\Sigma=(\iota,\sigma_0,\rho_0)$ be as in statement 3 of Fact \ref{qetest}. Since $\Sigma$ is an $\LL^{+}$-embedding, we get that $\a\in R_{\th}^{\AA}$ if and only if $\sigma_0(\a)\in R_{\th}^{\NN}$ and $\g\in R_{\psi}^{\AA}$ if and only if $\rho_0(\g)\in R_{\psi}^{\NN}$ for all $\LL_\kk\,$-formulas $\th(y),$ all $\LL_\GG\,$-formulas $\psi(z)$ and all suitable tuples $\alpha$ of $\kk(A)$ and tuples $\g$ of $\GG(A).$  

Hence, both $\sigma_0:\kk(A)\to\kk(N)$ and $\rho_0:\GG(A)\to\GG(N)$ are partial elementary maps with respect to $\LL.$ Indeed, if $\a$ is a tuple from $\kk(A)$ and $\g$ is a tuple from $\GG(A),$ then
    \begin{align*}
        \MM\models\theta(\a)&\Sii\MM\models R_{\th}(\a) & \MM\models\psi(\g)&\Sii\MM\models R_{\psi}(\g)\\
        &\Sii\NN\models R_{\th}(\sigma_0(\a)) &&\Sii\NN\models R_{\psi}(\rho_0(\g)) \\
        &\Sii\NN\models\th(\sigma_0(\alpha)), &&\Sii\NN\models\psi(\rho_0(\gamma))
    \end{align*}
for any $\LL_\kk\,$-formula $\th(y)$ and any $\LL_\GG\,$-formula $\psi(z).$ Since $\NN$ is $|M|^{+}$-saturated, we can extend the partial elementary maps $\sigma_0$ and $\rho_0$ to the whole of $\kk(M)$ and $\GG(M)$ respectively via a standard back-and-forth argument, i.e., there is some $\LL_\kk\,$-elementary map $\sigma:\kk(M)\to\kk(N)$ and some $\LL_\GG\,$-elementary map $\rho:\GG(M)\to\GG(N)$ such that $\sigma|_{\kk(A)}=\sigma_0$ and $\rho|_{\GG(A)}=\rho_0.$  
Note that $A$ is a $\l$-closed subring of $M,$ so $\langle A\rangle\subseteq\AA$ and $\sigma_0$ and $\rho_0$ is defined in $Av$ and $vA,$ respectively.  
By the Embedding Lemma \ref{emblem} applied to $\langle A\rangle$, we find some ring embedding $\widetilde{\iota}:M\to N$ such that $\widetilde{\Sigma}=(\widetilde{\iota},\sigma,\rho):\MM\to\NN$ is an $\LL$-embedding extending $\Sigma.$ Since $\sigma$ is $\LL_\kk\,$-elementary and $\rho$ is $\LL_\GG\,$-elementary, we get equivalently that $\widetilde{\Sigma}$ is an $\LL^{+}$-embedding, as wanted.    
\end{proof}

Unwinding the definitions, we get the following corollary.

\begin{cor}
\label{formbyform}
    Let $\LL=\LL(\LL_\kk,\,\LL_\GG)$ and let $x,y,z$ be tuples of variables of sort $\KK,\kk$ and $\GG,$ respectively. Any $\LL$-formula $\Phi(x,y,z)$ is equivalent, modulo $\texttt{SAMK}_{e}^{\l,\ac},$ to a boolean combination of the following formulas:
    \begin{enumerate}[wide, label*={\arabic*.}]
        \item $t(x)=0,$
        
        \item $\th(\underline{ac}(t_1(x)),\.,\underline{ac}(t_n(x)),y),$ where $\th$ is an $\LL_\kk\,$-formula, 

        \item $\psi(\underline{v}(t_1(x)),\.,\underline{v}(t_n(x)),z),$ where $\psi$ is an $\LL_\GG\,$-formula
    \end{enumerate}
    and $t(x),t_1(x),\.,t_n(x)$ are $\LL$-terms of sort $\KK^{|x|}\to\KK$. In particular, any $\LL$-sentence is equivalent, modulo $\texttt{SAMK}_{e}^{\l,\ac}$, to a boolean combination of sentences from $\LL_\kk$ and $\LL_\GG.$  
\end{cor}

\begin{cor}
\label{nipn}
Suppose $\LL=\LL(\LL_{ring},\,\LL_{og}\cup\{\infty\})$, and let $n\geq 1$. Any complete NIP$_n$ $\LL$-theory of infinite henselian equi-characteristic $\ac$-valued fields resplendently eliminates field quantifiers. In particular, the theory of any infinite positive characteristic NIP$_n$ $\ac$-valued field resplendently eliminates field quantifiers with respect to the language $\LL.$ 
\end{cor}
\begin{proof}
    Let $\MM$ be a NIP$_n$ henselian equi-characteristic $\ac$-valued field put as an $\LL_{ac}(\LL_{ring},\,\LL_{og}\cup\{\infty\})$-structure, and let $\MM^{\l}$ be the expansion of $\MM$ to an $\LL(\LL_{ring},\LL_{og}\cup\{\infty\})$-structure given by the interpretation of its parameterized $\l$-functions. By \cite[Theorem 3.5]{aj} for the case $n=1,$ or by \cite[Theorem 1.1]{b} for $n\geq1,$ we get that the underlying valued field $(M,v)$ of $\MM$ is separably-defectless Kaplansky and henselian, i.e. $\MM^{\l}\models\texttt{SAMK}_e^{\l,\ac}$ for some Ershov degree $e.$ By Proposition \ref{rqe}, the $\LL$-theory of $\MM^{\l}$ resplendently eliminates field quantifiers. In particular, by \cite[Theorem 2.8]{j} for $n=1$ and by \cite[Theorem 3.1]{hc} for $n\geq1,$ if $M$ is infinite and has positive characteristic, then $(M,v)$ is automatically henselian, and thus $\Th_\LL(\MM^{\l})$ resplendently eliminates field quantifiers as explained above.
\end{proof}

Now we turn our attention to existential formulas.

\begin{cor}
\label{delred}
Let $\LL=\LL(\LL_\kk,\,\LL_\GG)$ and let $x,y,z$ be tuples of variables of sort $\KK,\kk$ and $\GG,$ respectively. Any existential $\LL$-formula $\Phi(x,y,z)$ is equivalent, modulo $\texttt{SAMK}_{e}^{\l,\ac},$ to a $\{\vee,\wedge\}$-combination of the following formulas:
    \begin{enumerate}[wide, label*={\arabic*.}]
        \item $t(x)=0$ and their negations,
            
        \item $\th(\underline{ac}(t_1(x)),\.,\underline{ac}(t_n(x)),y),$ where $\th$ is an existential $\LL_\kk\,$-formula, 

        \item $\psi(\underline{v}(t_1(x)),\.,\underline{v}(t_n(x)),z),$ where $\psi$ is an existential $\LL_\GG\,$-formula
    \end{enumerate}
    and $t(x),t_1(x),\.,t_n(x)$ are $\LL$-terms of sort $\KK^{|x|}\to\KK$.
\end{cor}
\begin{proof}
    Let $\D=\E$ be the collection of all existential formulas from $\LL$, and let $\D'$ denote the collection of all $\{\vee,\wedge\}$-combinations of formulas of the form (1), (2) and (3). Note that $\D'$ contains all quantifier-free $\LL$-formulas and is closed under conjunctions and disjunctions.

    Let $\MM,\NN$ be models of $\texttt{SAMK}_{e}^{\l,\ac},$ and let let $a,\a,\g,b,\b,\d$ be finite tuples from $M,$ $\kk(M)$, $\GG(M)$, $N,$ $\kk(N)$ and $\GG(N)$ respectively. Suppose that $|a|=|b|,$ $|\a|=|\b|,$ and $|\g|=|\d|.$ 
    
    By Corollary \ref{delredeq}, it is enough to prove that $(\MM,a,\a,\g)\Rightarrow_{\E}(\NN,b,\b,\d)$ whenever 
    \begin{equation}
    \label{d'equiv}
        (\MM,a,\a,\g)\Rightarrow_{\D'}(\NN,b,\b,\d). 
    \end{equation}
    Let $\MM^{*}$ be an $\aleph_0$-saturated elementary extension of $\MM.$ Let $\AA$ be the $\LL$-substructure of $\MM$ generated by $a,\a$ and $\g$. In particular, $\langle A\rangle\subseteq\AA,$ where $A$ is the universe of the $\KK$-sort of $\AA.$ Since $\D'$ contains all quantifier-free $\LL$-formulas, then Equation \ref{d'equiv} implies that there is an $\LL$-embedding $\Sigma=(\iota,\sigma_0,\rho_0):\AA\to\NN$ mapping $(a,\a,\g)$ to $(b,\b,\d).$ 
    Also, since $\D'$ contains all formulas of the form (2) and (3), we get that $$(\kk(M^{*}),\xi)_{\xi\in Av}\Rightarrow_{\E}(\kk(N),\sigma_0(\xi))_{\xi\in Av}$$
    with respect to the corresponding expansion by constants of $\LL_\kk,$ and $$(\GG(M^{*}),\e)_{\e\in vA}\Rightarrow_{\E}(\GG(N),\rho_0(\e))_{\e\in vA}$$ with respect to the corresponding expansion by constants of $\LL_\GG.$ 
    If $\NN^{*}$ is an $|M^{*}|^{+}$-saturated elementary extension of $\NN,$ then Fact \ref{exemb} yields the existence of two functions $\sigma:\kk(M^{*})\to\kk(N^{*})$ and $\rho:\GG(M^{*})\to\GG(N^{*})$ preserving existential $\LL_\kk\,$- and $\LL_\GG\,$-formulas and such that $\sigma|_{Av}=\sigma_0$ and $\rho|_{vA}=\rho_0$ respectively. This is equivalent to saying that $\sigma$ and $\rho$ are $\LL_\kk\,$- and $\LL_\GG\,$-embeddings, respectively. 
    
    By the Embedding Lemma \ref{emblem}, there is some ring embedding $\widetilde{\iota}:M^{*}\to N^{*}$ inducing an $\LL$-embedding $\widetilde{\Sigma}=(\widetilde{\iota},\sigma,\rho):\MM^{*}\to\NN^{*}$ that extends $\Sigma.$ We can now conclude that $(\MM,a,\a,\g)\Rightarrow_{\E}(\NN,b,\b,\d),$ because if $\f(x,y,z)$ is an existential $\LL$-formula, then $$\MM\models\f(a,\a,\b)\implies\MM^{*}\models\f(a,\a,\b)\implies\NN^{*}\models\f(b,\b,\d)\implies\NN\models\f(b,\b,\d),$$
    where the second implication holds because $\f$ is existential, $\widetilde{\Sigma}$ has domain $\MM^{*}$, $\widetilde{\Sigma}(a,\a,\g)=(b,\b,\d)$ and $\widetilde{\Sigma}$ preserves quantifier-free formulas.
\end{proof}

\section{Ax-Kochen-Ershov Principles}

One of the differences between the following corollaries and the corresponding AKE-principles given in \cite[Theorem 6.3]{fvkp} lies the language, where in their case they work with an expansion by countably many predicates $\LL_Q=\LL_{3s}\cup\{Q_n:n\in\N^{>0}\},$ each to be interpreted across $n<\omega$ as the set of $p$-independent $n$-tuples. Also, our results do not require the imperfection degree of models to be finite.

\begin{cor}[AKE$^{\preceq},$ AKE$^{\equiv}$, AKE$^{\preceq_\E},$ AKE$^{\equiv_\E}$]
    Let $\LL=\LL(\LL_\kk,\,\LL_\GG)$ and let $\MM,\NN\models \texttt{SAMK}_{e}^{\l,\ac}.$ Then: 
    \begin{enumerate}[wide, label*={\arabic*.}]
        \item If $\MM\subseteq\NN,$ then $\MM\preceq\NN$ if and only if $\GG(M)\preceq\GG(N)$ as $\LL_\GG\,$-structures and $\kk(M)\preceq\kk(N)$ as $\LL_\kk\,$-structures.

        \item $\MM\equiv\NN$ if and only if $\GG(M)\equiv\GG(N)$ as $\LL_\GG\,$-structures and $\kk(M)\equiv\kk(N)$ as $\LL_\kk\,$-structures.

        \item If $\MM\subseteq\NN,$ then $\MM\preceq_\E\NN$ if and only if $\GG(M)\preceq_\E\GG(N)$ as $\LL_\GG\,$-structures and $\kk(M)\preceq_\E\kk(N)$ as $\LL_\kk\,$-structures.
    
        \item $\MM\equiv_\E\NN$ if and only if $\GG(M)\equiv_\E\GG(N)$ as $\LL_\GG\,$-structures and $\kk(M)\equiv_\E\kk(N)$ as $\LL_\kk\,$-structures.
    \end{enumerate}
\end{cor}
\begin{proof}
    All four direct implications are clear. We prove their converse implications. 
    \begin{enumerate}[wide,  label*={\arabic*.}]
        \item Suppose that $\MM\subseteq\NN,$ that $\kk(M)\preceq\kk(N)$ and that $\GG(M)\preceq\GG(M)$. By Corollary \ref{formbyform}, the $\LL$-embedding given by $\Sigma=(\id_M,\id_{\kk(M)},\id_{\GG(M)})$ is in fact an elementary embedding, implying that $\MM\preceq\NN,$ as wanted.

        \item Suppose that $\kk(M)\equiv\kk(N)$ and $\GG(M)\equiv\GG(M).$ By Corollary \ref{formbyform}, any $\LL$-sentence is equivalent, modulo $\texttt{SAMK}_{e}^{\l,\ac},$ to a boolean combination of sentences from $\LL_\kk$ and $\LL_\GG,$ so the hypotheses yield that $\MM\equiv\NN,$ as wanted.

        \item Suppose that $\MM\subseteq\NN,$ that $\kk(M)\preceq_\E\kk(N)$ and that $\GG(M)\preceq_\E\GG(N)$. By Corollary \ref{delred}, the $\LL$-embedding given by $\Sigma=(\id_M,\id_{\kk(M)},\id_{\GG(M)})$ is in fact an existentially closed embedding, implying that $\MM\preceq_\E\NN,$ as wanted.
        
        \item Suppose that $\kk(M)\equiv_\E\kk(N)$ and $\GG(M)\equiv_\E\GG(M).$ By Corollary \ref{delred}, any existential $\LL$-sentence is equivalent, modulo $\texttt{SAMK}_{e}^{\l,\ac},$ to a $\{\vee,\wedge\}$-combination of existential sentences from $\LL_\kk$ and $\LL_\GG,$ so the hypotheses yield that $\MM\equiv_\E\NN,$ as wanted.    
        \qedhere
    \end{enumerate}
\end{proof}

\begin{cor}[Completions]
\label{completions}
    The completions of $\texttt{SAMK}_{e}^{\l,\ac}$ are $\texttt{SAMK}_{e}^{\l,\ac}\cup\Th_{\LL_\kk}(k)\cup\Th_{\LL_\GG}(\G),$ where $k$ is a $p$-closed field and $\G$ is a $p$-divisible ordered abelian group for some characteristic exponent $p.$
\end{cor}
\begin{proof}
    The result follows from the AKE$^{\equiv}$ principle above.
\end{proof}

\begin{cor}
    Let $\MM\models\texttt{SAMK}_{e}^{\l,\ac}.$ Then $\kk(M)$ and $\GG(M)$ are orthogonal and stably-embedded in $\MM.$ 
\end{cor}
\begin{proof}
    Let $n,m\in\N$ and let $X\subseteq\kk(M)^{n}\times\GG(M)^{m}$ be an $\LL$-definable set. By Corollary \ref{formbyform}, $X$ is the trace in $\MM$ of a boolean combination of formulas without free variables from $\KK$ of the form $\th(y,\a)\wedge\psi(z,\g),$ where $y$ is an $n$-tuple of variables of sort $\kk,$ $z$ is an $m$-tuple of variables of sort $\GG$, $\a$ is a tuple of elements from $\kk(M),$ $\g$ is a tuple of elements from $\GG(M),$ $\th$ is an $\LL_\kk\,$-formula and $\psi$ is an $\LL_\GG\,$-formula. This proves orthogonality. Note that we can dismiss formulas of the form $t(a)=0$ or $t(a)\neq0$ appearing in the definition of $X$ and where $a$ is a finite tuple of elements of $M$ and $t$ is an $\LL$-term of sort $\KK^{|a|}\to\KK,$ because we can replace them, say, by the $\LL_\kk\,$-formulas $0=0$ or $0=1$ if $t^{\MM}(a)=0$ or $t^{\MM}(a)\neq0$ respectively, and still obtain a formula defining $X.$  
    
    If $X\subseteq \kk(M)^{n}$ is an $\LL$-definable set, then the same argument above proves that $X$ is the trace in $\MM$ of a formula of the form $\th(y,\a)$ where $y,\a$ and $\th$ are as above, proving stable embeddedness of $\kk(M)$ in $\MM.$ Once again, we may dismiss formulas of the form $t(a)=0$ or of the form $\psi(\g),$ where $a$ is a tuple of elements of $M,$ $t$ is an $\LL$-term of sort $\KK^{|a|}\to\KK,$ $\psi$ is an $\LL_\GG\,$-formula and $\g$ is a tuple of elements of $\GG(M).$ Mutatis mutandis, we get that $\GG(M)$ is stably embedded in $\MM$ as well.    
\end{proof}

Finally, we turn our attention to Anscombe's and Fehm's treatment of decidability and computability of theories of valued fields as presented in \cite{af}, in order to get the following corollaries about countable theories of $\texttt{SAMK}_e^{\l,\ac}$ valued fields.

\begin{cor}
\label{mtosamk}
    Let $p$ be either 1 or a prime number, let $T_{\kk}$ be an $\LL_\kk\,$-theory of $p$-closed fields of characteristic exponent $p$, let $T_{\GG}$ be an $\LL_\GG\,$-theory of $p$-divisible ordered abelian groups, and let $e\in\N\cup\{\infty\}.$
    Suppose that $\LL=\LL(\LL_\kk,\,\LL_\GG)$ is countable and $T=\mathcal{T}(e,T_{\kk},T_{\GG})$ is computably enumerable. Then 
    \begin{enumerate}[wide, label*={\arabic*.}]
        \item $T\leq_T T_{\kk}\oplus T_{\GG}.$  

        \item If $T_\kk$ and $T_\GG$ are complete, then $T_\E\leq_T (T_{\kk})_\E\oplus (T_{\GG})_\E.$ 
    \end{enumerate}
\end{cor}
\begin{proof}
    \begin{enumerate}[wide,  label*={\arabic*.}]
        \item Let $\D$ be the full set of $\LL$-formulas, and let $\D'$ be the set of boolean combinations of formulas of the form (2) and (3) of Corollary \ref{formbyform}. 
        These sets of formulas are easily seen to be computable. In this case $T_{\D}=T$, so $T\leq_mT_{\D'}$ by Corollary \ref{mto} and Corollary \ref{formbyform}. It follows in particular that $T\leq_TT_{\D'}.$ 
        Now we have to argue that $T_{\D'}\leq_T T_\kk\oplus T_\GG.$ Indeed, given sentence $\f\in T_{\D'},$ there are some $n\in\N,$ some $\LL_\kk\,$-sentences $\th_1,\.,\th_n$ and some $\LL_\GG\,$-sentences $\psi_1,\.,\psi_n$ such that $\vdash\f\sii\bigwedge_{i=0}^{n}(\th_i\vee\psi_i),$ and moreover there is an algorithm that computes the sentences $\{\th_1,\.,\th_n,\psi_1,\.,\psi_n\}$ when given $\f$ as an input. It follows that $T\vdash\f$ if and only if $T\vdash\th_i\vee\psi_i$ for each $i\in\{1,\.,n\}.$ By Item 1 of Corollary \ref{disj}, for each $i\in\{1,\.,n\},$ we get that $T\vdash\th_i\vee\psi_i$ if and only if $T_\kk\vdash\th_i$ or $T_{\GG}\vdash\psi_i,$ which are questions we can ask to the oracles of $T_\kk$ and $T_{\GG}.$ The result follows.

        \item Let $\D=\E$ be the set of existential $\LL$-formulas, and let $\D'$ be the set of $\{\vee,\wedge\}$-combinations of formulas of the form (1), (2) and (3) of Corollary \ref{delred}. Since $\D$ and $\D'$ are again computable, following the same argument as above we get that $T_\E\leq_T T_{\D'}.$ As above, mutatis mutandis, and using Item 2 of Corollary \ref{disj}, we obtain $T_{\D'}\leq_T(T_\kk)_\E\oplus(T_\GG)_\E,$ as wanted.
        \qedhere
    \end{enumerate} 

\end{proof}

\begin{cor}[Decidability]
    Suppose that $\LL=\LL(\LL_\kk,\,\LL_\GG)$ is countable. Let $p$ be $1$ or a prime number, let $T_{\kk}$ be a computably enumerable $\LL_\kk\,$-theory of $p$-closed fields of characteristic exponent $p$, let $T_{\GG}$ be a computably enumerable $\LL_\GG\,$-theory of $p$-divisible ordered abelian groups, and let $e\in\N\cup\{\infty\}.$ Let $T$ denote the $\LL$-theory $\mathcal{T}(e,T_\kk,T_\GG).$
    \begin{enumerate}[wide, label*={\arabic*.}]
        \item $T$ is decidable if and only if so are $T_\kk$ and $T_\GG$.

        \item If $T_\kk$ and $T_{\GG}$ are complete, then $T_\E$ is decidable if and only if so are $(T_\kk)_\E$ and $(T_\GG)_\E$.
    \end{enumerate}
\end{cor}
\begin{proof}
This follows from Corollary \ref{mtosamk} and the well known fact that if $A\leq_T B$ and $B$ is computable, then $A$ is computable. 
\end{proof}

\subsection*{Acknowledgments}
The author would like to thank his supervisor Sylvy Anscombe for encouraging him to pursue this project and for all her valuable advice, Arno Fehm and Franziska Jahnke for giving valuable suggestions on the subject and the writing of the document, and the latter for hosting the author at the Mathematics Cluster of Excellence of the University of Münster during the preparation of this article. Special thanks are due to Franz-Viktor Kuhlmann, whose suggestions allowed the author to improve early versions of this manuscript. The author would also like to thank the Collège des Écoles Doctorales and the SMARTS-UP project of Université Paris Cité, which made his visit to Münster possible under the BDMI mobility program. The author has received funding from the European Union’s Horizon 2020 research and innovation program under the Marie Sk\lw odowska-Curie grant agreement No 945332. 
\includegraphics*[scale = 0.028]{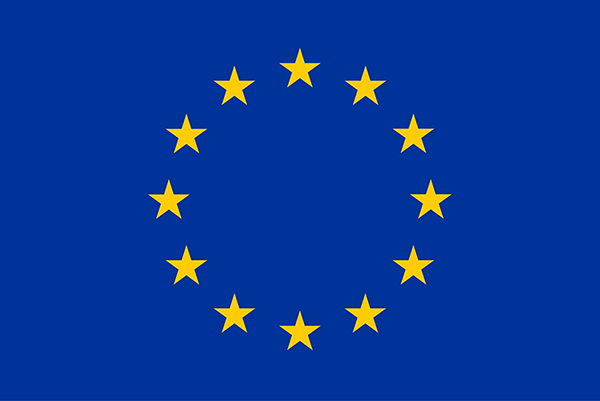}

\bibliographystyle{alpha}
\bibliography{bib}

\end{document}